\documentclass{article}

\usepackage[nonatbib, preprint]{neurips_2021}
\usepackage[sort&compress,numbers]{natbib}
\usepackage[utf8]{inputenc} % allow utf-8 input
\usepackage[T1]{fontenc}    % use 8-bit T1 fonts
\usepackage{url}            % simple URL typesetting
\usepackage{booktabs}       % professional-quality tables
\usepackage{amsfonts}       % blackboard math symbols
\usepackage{nicefrac}       % compact symbols for 1/2, etc.
\usepackage{microtype}      % microtypography
\usepackage{xcolor}         % colors
\usepackage{verbatim} 

\usepackage{microtype}
\usepackage{graphicx}
\usepackage{caption}
\usepackage{subcaption}
\usepackage{booktabs} % for professional tables
\usepackage{comment}
\usepackage{wrapfig}
\usepackage{enumerate}
\usepackage[shortlabels]{enumitem}

\usepackage[colorlinks,
linkcolor=red,
anchorcolor=blue,
citecolor=blue
]{hyperref}
\usepackage[noend]{algpseudocode}

\usepackage{smile, bbm}

\usepackage{enumitem}
\usepackage{tikz}

\title{FedHybrid: A Hybrid Primal-Dual Algorithm Framework for Federated Optimization}

\author{%
  Xiaochun Niu \\
  Department of Industrial Engineering and Management Sciences\\
  Northwestern University\\
  \texttt{xiaochunniu2024@u.northwestern.edu} \\
  \And
  Ermin Wei \\
  Department of Industrial Engineering and Management Sciences\\
  Northwestern University\\
  \texttt{ermin.wei@northwestern.edu} \\
}

\begin{document}
\maketitle
% Date 04/29/2021
% In this file, we reorganize the proof in version_6.

%!TEX root = nips_submit/main.tex

\begin{abstract}
We consider a multi-agent consensus optimization problem over a server-client (federated) network, where all clients are connected to a central server. Current distributed algorithms fail to capture the heterogeneity in clients' local computation capacities. Motivated by the generalized Method of Multipliers in centralized optimization, we derive {an approximate} Newton-type primal-dual method with a practical distributed implementation by utilizing the server-client topology. Then we propose a new primal-dual algorithm framework \textit{FedHybrid} that allows different clients to perform various types of updates. Specifically, each client can choose to perform either gradient-type or Newton-type updates. We propose a novel analysis framework for primal-dual methods and obtain a linear convergence rate of FedHybrid for strongly convex functions, regardless of clients' choices of gradient-type or Newton-type updates. Numerical studies are provided to demonstrate the efficacy of our method in practice. To the best of our knowledge, this is the first hybrid algorithmic framework allowing heterogeneous local updates for distributed consensus optimization with a provable convergence and rate guarantee.
\end{abstract}

\section{Introduction}

The problem of optimizing an objective function by employing a distributed procedure over a network where both data collection and model training is pushed to massive edge clients has gained significant attention recently. This is motivated by the large-scale nature of many modern big-data problems such as multi-vehicle and multi-robot networks \citep{cao2012overview,zhou2011multirobot}, machine learning \cite{duchi2011dual,tsianos2012consensus}, and especially \textit{federated learning} \cite{konevcny2016federated}. In such problems, a set of $n$ clients are connected to a central node (server), where each client has access only to its local data. {We refer to such a model as a \textit{server-client (federated) network.}} The goal is to learn a shared model over all the data in the network without exchanging local data due to privacy issues or communication limitations \citep{konevcny2016federated}. Formally, if we define $f_i:\RR^d\to \RR$ as local objective function corresponding to client $i$, the system-wide problem can be represented as
\#\label{prob:original}
    \min_{\omega\in \RR^d} \quad\sum\limits_{i=1}^n f_i(\omega).
\#
For instance, in a supervised learning setting, consider an empirical risk minimization problem, the local objective function $f_i$ represents expected loss over the local data distribution of client $i$.

To develop a distributed computation method where each client only has access to local information and can only communicate with the central server, we first decouple the computation of individual client by introducing local copies of the decision variable. We index the central server as the $0^{th}$ node and denote by $x_0, x_i \in \RR^d$ the local copies of $\omega$ kept at the central server and each client $i\in[n]$, respectively. Problem \ref{prob:original} is reformulated as the following \textit{consensus optimization} problem \citep{bertsekas1989parallel,nedic2009distributed},
\#\label{prob:constrained}
    \min_{\tilde{x}}  \quad \tilde{f}(\tilde{x}) = \sum\limits_{i=1}^n f_i(x_i) \quad
    \text{s.t. }W\tilde{x} = 0, 
\#
where $x = \rbr{x_1^\intercal,\cdots,x_{n}^\intercal}^\intercal \in \RR^{nd}$ and $\tilde{x} = \rbr{x_0^\intercal,x^\intercal}^\intercal \in \RR^{(n+1)d}$ are concatenations of local variables, $\tilde{f}:\RR^{(n+1)d}\to\RR$ is the aggregate function and $W = \rbr{\mathbbm{1}_n, -I_n}\otimes I_d \in \RR^{nd \times (n+1)d}$ is the adjacency matrix of the server-client network, where the $i^{th}$ row of $W\tilde x = 0$ represents the consensus constraint $x_0 = x_i$ and enforces the equivalence of the two problem. For notational simplicity, we define $f:\RR^{nd}\to\RR, f(x) = \tilde{f}(\tilde{x})$. {Also, throughout the paper, }{we differentiate whether a variable includes the server's decision or clients' only by using a tilde.}

While there is a proliferating literature on developing distributed optimization methods to solve the above problem, most existing works are either gradient-type methods \citep{nedic2009distributed,shi2015extra,jakovetic2014fast} or Newton-type methods \citep{shamir2014communication,zhang2015disco,wang2018giant,crane2019dingo}, where the clients have no flexibility to choose the types of updates. In stark comparison, due to the drastically varying capabilities of storage, computation, and communication among clients caused by hardware, network connectivity, and battery power, distributed systems in practice may involve severely heterogeneous clients \citep{chen2015mxnet, li2020federated}. The desire to provide hybrid methods for networks involving heterogeneous clients is even more pronounced in the \textit{federated learning} setting \citep{konevcny2016federated,li2018federated,li2020federated}. 
As a result, it is imperative to provide a flexible and efficient hybrid algorithm framework to tolerate heterogeneous clients and utilize the heterogeneous structure. To the best of our knowledge, however, algorithms with heterogeneous clients setting have not been considered before. 

% leaving a huge gap between distributed optimization methods and heterogeneous clients

To incorporate heterogeneous clients, {we propose a hybrid primal-dual algorithm framework named \textit{FedHybrid}, which is a distributed method for server-client (federated) network allowing some clients to do first-order updates while others use second-order information.}  %FedHybrid provides enough flexibility when implementing the algorithm in practice, especially in cases with heterogeneous clients. 
Specifically, motivated by the generalized Method of Multipliers (MM) in centralized optimization \citep{tapia1977diagonalized,bertsekas2014constrained}, we first introduce a gradient-type primal-dual approximation of MM with a practical distributed implementation. Second, for a speedup, we propose a Newton-type approximation of MM, where we derive a distributed approximated dual Hessian utilizing the server-client network. Then we propose FedHybrid as a combination of gradient-type and Newton-type primal-dual methods. In specific, those clients with higher computational capabilities and/or cheaper cost to perform computation can implement Newton-type updates locally, while other clients can adopt much simpler gradient-type updates. Finally, we propose a novel analysis framework for primal-dual algorithms and obtain a last iterate linear (Q-linear) convergence rate of FedHybrid for strongly convex objective functions. Numerical experiments on both synthetic data and real-life data are conducted to demonstrate the efficacy of our method.

\vskip5pt
\noindent\textbf{Contributions.} Our contribution is threefold. First, we propose a Newton-type primal-dual method with a practical distributed implementation by approximating primal and dual Hessian utilizing the server-client network. Second, we propose FedHybrid as a primal-dual method combining both gradient-type and Newton-type updates for heterogeneous clients. Finally, we propose a novel analysis framework and obtain a last iterate linear convergence rate of FedHybrid for strongly convex objective functions. Numerical experiments are provided to demonstrate the efficacy of FedHybrid in practice. 
To the best of our knowledge, this is the first hybrid algorithmic framework allowing heterogeneous local updates for distributed consensus optimization with a provable convergence guarantee.
 
\vskip5pt
\noindent\textbf{Related Work.} Our work is closely related to the growing literature on distributed algorithms for solving Problem \ref{prob:constrained}. 
We outline various lines of researches as follows. We start by reviewing the literature studying general network topology. Primal iterative methods, including distributed (sub)-gradient descent (DGD) and related methods \cite{nedic2009distributed, shi2015extra}, have updates in the form of some linear combinations of a gradient descent step with respect to its local objective function and a weighted average with local neighbors. A related work is \citep{9069966}, where the authors derive a DGD based method with the inclusion of first and second order update in continuous time setting, which cannot be directly translated to a discrete time update. While they provide an asymptotic convergence guarantee, their work lacks convergence rate analysis.
There are econd-order methods including Network Newton \citep{7094740} and Distributed Newton method \citep{tutunov2019distributed}. Also, dual decomposition based methods include ADMM \citep{boyd2011distributed,wang2019global}, ESOM \citep{mokhtari2016decentralized}, and CoCoA\citep{smith2018cocoa}. While none of these works consider a mixture of first and second order updates, ESOM with an approximated primal-dual framework is closely related to our work. The main difference there is ESOM uses only gradient-type update for the dual variables. Since we utilize the server-client topology, FedHybrid allows dual updates to also mimic a Newton step.

Another line of research focuses on the server-client network or the federated learning setting. Existing works include primal first-order methods like FedAvg \citep{mcmahan2017communication,li2019convergence}, FedProx\citep{li2018federated} and FedAC \citep{yuan2020federated}. There are also primal-dual methods like FedPD \citep{zhang2020fedpd} that utilizing first-order updates in both primal and dual space. Such methods suffer from slow convergence due to their first-order nature. %With the second-order updates involving, FedHybrid can speedup the convergence rate.
Some second-order methods have been proposed. Examples include DANE \citep{shamir2014communication}, DiSCO \citep{zhang2015disco}, GIANT \citep{wang2018giant}, and DINGO \citep{crane2019dingo}. However, most of these methods require the assumption that each client has access to IID local data, which is not realistic in practice, especially in federated learning settings. With the primal-dual framework, FedHybrid can handle non-IID data distributions among clients and enjoy the speedup brought by second-order updates.

\vskip5pt
\noindent\textbf{Notations.} For any integer $m$, we denote by $[m] = \{ 1,\cdots,m\}$, $I_m \in \RR^{m\times m}$ the identity matrix, and $\mathbbm{1}_m= \rbr{1,\cdots,1}^\intercal\in \RR^{m}$. Also, we denote by $\otimes$ the Kronecker product.

%!TEX root = main.tex

\section{Preliminaries}
In this section, we review the generalized Method of Multipliers (MM) in centralized optimization that is derived by formulating a dual problem {based on augmented Lagrangian} \citep{bertsekas2014constrained}. The MM method helps to motivate our derivation of FedHybrid.
Before presenting the methods, we first introduce {a standard condition on local objective functions that we will assume to hold throughout the paper.}
\begin{assumption}
\label{ass:hessian}
For all $i\in[n]$, the local function $f_i(\cdot)$ is twice continuously differentiable and the eigenvalues of the Hessian $\nabla^2f_i(\cdot)$ are bounded by constants $0 < {m_i} \le \ell_i < \infty$, that is, ${m_i} I \preceq  \nabla^2f_i(\cdot)  \preceq {\ell_i} I$. 
\end{assumption}
For all $i \in[n]$, the above upper bound {on the Hessian} implies that the gradient $\nabla f_i$ is $\ell_i$-Lipschitz continuous and the lower bound implies that the function $f_i$ is $m_i$-strongly convex. For notational convenience, we denote by $m = \min_{i\in[n]}\{m_i\}$ and $\ell = \max_{i\in[n]}\{\ell_i\}$.

\noindent\textbf{Dual Problem.}
{Now we introduce a dual problem  {of  Problem \ref{prob:constrained} based on augmented Lagrangian}.} To do so, we
introduce dual variables $\lambda = \rbr{\lambda_1^\intercal,\cdots, \lambda_n^\intercal}^\intercal\in \RR^{nd}$ with $\lambda_i\in\RR^d$ associated with the $i^{th}$ constraint $x_0=x_i$ and define the augmented Lagrangian function $\tilde{L}(\tilde{x}, \lambda)$ of Problem \ref{prob:constrained} as 
\#\label{eq:lagrangian_func}
\tilde{L}(\tilde{x},\lambda) = \tilde{f}(\tilde{x}) + \lambda^{\intercal}W\tilde{x} + \dfrac{\mu}{2}\tilde{x}^{\intercal}W^{\intercal}W\tilde{x},
\# 
where $\mu>0$ is a constant. We remark that the matrix $W^{\intercal}W$ is the graph Laplacian of the server-client network. The augmentation term $\mu \tilde{x}^{\intercal}W^{\intercal}W\tilde{x} / 2$ is {zero} if $\tilde{x}$ is feasible to Problem \ref{prob:constrained}; otherwise, it is positive and serves as a penalty for the violation of the consensus constraint. Then the dual problem is defined as follows, 
\#\label{prob:dual}
\max_{\lambda\in\RR^{nd}} g(\lambda), \quad \text{where } g(\lambda) = \min_{\tilde{x}\in\RR^{(n+1)d}} \tilde{L}(\tilde{x}, \lambda).
\#
{We will refer to $g:\RR^{nd}\to\RR$ as the dual function.}  For any $\lambda\in\RR^{nd}$, we define $\tilde{x}^*(\lambda)= \argmin_{\tilde{x}} \tilde{L}(\tilde{x},\lambda)$ as the optimal point of the inner problem in Problem \ref{prob:dual}. Following from Assumption \ref{ass:hessian} and {Slater’s condition, the strong duality holds \cite{boyd2004convex}.}
%, that is,} there is no duality gap between the primal Problem \ref{prob:constrained} and its dual Problem \ref{prob:dual}.
Thus, $\tilde{x}^*(\lambda^*)$ is the optimal solution of Problem \ref{prob:constrained}, where $\lambda^*$ is {an} optimal solution of Problem \ref{prob:dual}.

 %{Should we call it MM or ADMM here?}
\noindent\textbf{Generalized MM.} We introduce the generalized MM \citep{tapia1977diagonalized,bertsekas2014constrained} {that uses the augmented Lagrangian function $\tilde{L}$ defined in \eqref{eq:lagrangian_func} to solve Problem \ref{prob:constrained} as follows.} 
% The generalized MM  \citep{tapia1977diagonalized,bertsekas2014constrained} used to solve Problem \ref{prob:dual} with the augmented Lagrangian function $\tilde{L}$ defined in \eqref{eq:lagrangian_func} is stated as follows. 
At each iteration $k$, 
 \#
& \tilde{x}^*(\lambda^k) = \argmin_{\tilde{x}} \tilde{L}(\tilde{x},\lambda^k),\quad \lambda^{k+1} = U(\tilde{x}^*(\lambda^k), \lambda^k), \label{eq:mm1}
 \#
where $U:\RR^{(n+1)d\times nd}\to\RR$ is a {general} dual update formula with the property that $\lambda^* = U(\tilde{x}^*(\lambda^*), \lambda^*)$.
We give some popular choices for $U$ as follows, 
\#
U_1(\tilde{x}^*(\lambda^k),\lambda^k) &= \lambda^k + \beta_1 \nabla g(\lambda^k), \label{eq:u1}\\
U_2(\tilde{x}^*(\lambda^k),\lambda^k) &= \lambda^k - \beta_2 \rbr{\nabla^2 g(\lambda^k)}^{-1}\nabla g(\lambda^k), \label{eq:u2} 
\#
where $\beta_1, \beta_2 > 0$ are stepsizes. We remark that \eqref{eq:u1} and \eqref{eq:u2} correspond to gradient ascent method and Newton's method {with respect to the dual function defined in \ref{prob:dual}}, respectively. 

{The following lemma gives the explicit forms of $\nabla g(\lambda)$ and $\nabla^2 g(\lambda)$ used in \eqref{eq:u1} and \eqref{eq:u2}. 
\begin{lemma}[\cite{tapia1977diagonalized,tutunov2019distributed}] \label{lem:dual_hessian}
Under Assumption \ref{ass:hessian}, the gradient and the Hessian of the dual function defined in Problem \eqref{prob:dual} are given by the following formulas, respectively. \$\nabla g(\lambda) = W\tilde{x}^*(\lambda), 
\quad \nabla^2 g(\lambda) =  -W\rbr{\nabla_{\tilde{x}\tilde{x}}^2\tilde L(\tilde{x}^*(\lambda),\lambda)}^{-1}W^\intercal.\$
\end{lemma}}

We remark that the primal update in \eqref{eq:mm1} is computationally expensive due to the requirement of an exact solution to the inner minimization problem and it cannot be {readily} implemented in a distributed manner due to the nonseparable augmentation term  $\mu\tilde{x}^\intercal W^\intercal W\tilde{x}/2$. %{Thus, in the next section, we introduce some types of primal-dual approximation of MM.}

\section{Algorithm}
In this section, to reduce the computational complexity and obtain an update in a distributed manner, we first derive a gradient-type approximation of MM, which recovers the Arrow-Hurwicz-Uzawa method \citep{arrow1958h}. This is a special case of FedHybrid when all clients perform gradient-type updates. Then for a speedup, we propose a Newton-type approximation of MM utilizing the server-client topology, which leads to another special case of FedHybrid with all clients performing Newton-type updates. Finally, we combine the gradient-type and Newton-type updates to provide our FedHybrid method, which allows a mixture of first and second order updates by various clients and hence is able to provide flexibility to handle heterogeneity in the network.

\subsection{Gradient-type Approximation of MM}
To develop a first-order method based on the dual gradient ascent update in \eqref{eq:u1}, we need to compute  $\nabla g(\lambda^k) = W\tilde x^*(\lambda)$ by Lemma \ref{lem:dual_hessian}. However, the computation of an exact minimizer $\tilde x^*(\lambda^k)$ can be computationally expensive. Thus, we approximate it by taking a primal gradient descent step on the primal variable and use the obtained $\tilde{x}^k$ as an approximation. This leads to the following gradient-type primal-dual algorithm. At each iteration $k$,
\#\label{eq:up_1st}
\tilde{x}^{k+1} = \tilde{x}^k - \alpha_1\nabla_{\tilde{x}}\tilde{L}(\tilde{x}^k,\lambda^k),  \quad \lambda^{k+1} =\lambda^k + \beta_1W\tilde{x}^k,
\#
where $\alpha_1, \beta_1 >0$ are primal and dual stepsizes, respectively. {This recovers the} Arrow-Hurwicz-Uzawa method \citep{arrow1958h}, which is {a special case of} FedHybrid when all clients perform gradient-type updates. %The updates in \eqref{eq:up_1st} 
%can be readily implemented in a distributed scheme. 
%In each iteration $k$, at the central server, $x^{k+1}_0 = (1-\alpha_1\mu n)x^k_0-  \alpha_1( \sum_{i\in[n]}\lambda_i^k -\mu \sum_{i\in[n]}x_i^k)$; at each client $i\in[n]$,
%\$
%x_i^{k+1} = x_i^k - \alpha_1 \rbr{\nabla f_i(x_i^k) - \lambda_i^k + \mu(x_i^k - x_0^k)}, \quad
%\lambda_i^{k+1} = \lambda_i^k + \beta_1 \rbr{x_0^k - x_i^k}.
%\$
{While this gradient-type approximation lead to simple distributed implementation, it suffers from slow convergence due to its first-order nature}.  {This motivates us to consider Newton's method for a speedup.} %Next, we will propose a Newton-type approximation of MM under the server-client topology.

\subsection{Newton-type Approximation of MM} 
We now derive a Newton-type approximation of the generalized MM under the server-client topology. 

\noindent\textbf{Primal Update.} 
We consider a Newton's step as an approximation of the primal update in \eqref{eq:mm1}.  Note that the primal Hessian $\nabla_{\tilde{x}\tilde{x}}^2\tilde L(\tilde{x}^k,\lambda^k) = \nabla^2\tilde{f}(\tilde{x}^k) + \mu W^\intercal W$ is nonseparable due to the graph Laplacian $W^\intercal W$, which makes the computation of the exact Hessian inverse intractable in a decentralized setting. Thus, we approximate the graph Laplacian $W^\intercal W$ by its block diagonal part $E = \diag \{nI_d, I_d,\cdots, I_d\}\in \RR^{(n+1)d\times(n+1)d}$. By using $\tilde{H} = \nabla^2\tilde{f}(\tilde{x}) + \mu E$ to approximate the primal Hessian $\nabla_{\tilde{x}\tilde{x}}^2\tilde L(\tilde{x},\lambda)$, we obtain the following Newton-type primal update. At each iteration $k$,
\#\label{eq:xp1}
\tilde{x}^{k+1} = \tilde{x}^k - (\tilde{H}^k)^{-1}\nabla_{\tilde{x}}\tilde{L}(\tilde{x}^k,\lambda^k), 
\#
where $\tilde H^k = \nabla^2\tilde{f}(\tilde{x}^k) + \mu E$ . In particular, the server update takes the following form,  
\#\label{eq:xp2}
x_0^{k+1} = \frac{1}{n}\sum_{i\in[n]}x_i^{k} - \frac{1}{\mu n}\sum_{i\in[n]}\lambda_i^{k}.
\#
We remark that $x_0^{k+1}$ can also be viewed as a penalized average of the primal decision variables $\{x_i^{k}\}_{i\in [n]}$, where the second term in \eqref{eq:xp2} serves as the penalty.

\noindent\textbf{Dual Update.} 
We define the exact dual Newton update as $\Delta \lambda^k$, that is, $\Delta \lambda^k$ satisfies the Newton update formula $\nabla^2 g(\lambda^k) \Delta \lambda^k = \nabla g(\lambda^k)$. 
Note that the explicit form of the Hessian $\nabla^2 g(\lambda^k)$ given in Lemma \ref{lem:dual_hessian} is nonseparable, making it difficult to compute $\Delta \lambda^k$ in a distributed manner. Thus, to obtain a Newton-type dual update in the form of \eqref{eq:u2} in a distributed scheme, we will provide an approximation of $\Delta \lambda^k$. We first define $\hat \nabla g(\lambda^k)$ and $\hat \nabla^2 g(\lambda^k)$ as estimators of $\nabla g(\lambda^k)$ and $\nabla^2 g(\lambda^k)$ defined in Lemma \ref{lem:dual_hessian}, respectively, as follows 
\#\label{eq:app_na}
\hat \nabla g(\lambda^k) = W\tilde{x}^k, \quad \hat\nabla^2 g(\lambda^k) = -W\rbr{\nabla_{\tilde{x}\tilde{x}}^2\tilde L(\tilde{x}^k,\lambda^k)}^{-1}W^\intercal,
\#
where we substitute  $\tilde x^k$ as an approximation of $\tilde{x}^*(\lambda^k)$. 
Following from \eqref{eq:app_na}, we introduce the following lemma to provide an approximation $\Delta \breve{\lambda}^k$ of $\Delta \lambda^k$, where $\Delta \breve{\lambda}^k$ satisfies the approximated dual Newton update $\hat \nabla^2 g(\lambda^k) \Delta \breve\lambda^k = \hat \nabla g(\lambda^k)$. The proof is deferred to Section \ref{sect:a1}.  
\begin{lemma} \label{lem:dual_hessian_app}
Under Assumption \ref{ass:hessian}, the approximation of dual Newton update $\Delta \breve{\lambda}^k$ defined as above satisfies
\$
W^\intercal \Delta \breve{\lambda}^k = \nabla^2_{\tilde{x}\tilde{x}} \tilde{L}(\tilde{x}^k, \lambda^k)(\mathbbm{1}_{n+1}\otimes y^k - \tilde{x}^k),
\$
where $y^k = \big(\sum_{i\in[n]}\nabla^2f_i(x_i^k) \big)^{-1}\sum_{i\in[n]}\nabla^2f_i(x_i^k) x_i^k$ is the Hessian weighted average of the primal decision variables. 
\end{lemma}

If we further substitute $x_0^{k+1}$ defined in \eqref{eq:xp2} and $\tilde{H}^k$ defined above as approximations of $y^k$ and $\nabla^2_{\tilde{x}\tilde{x}} \tilde{L}(\tilde{x}^k, \lambda^k)$, respectively, we obtain an approximation $\Delta\hat{\lambda}^k$ of $\Delta \breve{\lambda}^k$ satisfying
$
W^\intercal\Delta\hat{\lambda}^k = \tilde{H}^k(\mathbbm{1}_{n+1}\otimes x_0^{k+1} - \tilde{x}^k).
$ 
Then following from the structure of the adjacency matrix $W$, we have 
\#\label{eq:hat_lam}
\Delta\hat{\lambda}^k = -H^k W\tilde{x}^k,
\#
where $H^k = \nabla^2 f(x^k) + \mu I_{nd}\in\RR^{nd\times nd}$, which is a submatrix of $\tilde H^k$, corresponding to the components related to the clients. %We remark that $\tilde{H}^k = \diag \{\mu nI_d, H^k\}$.
Thus, using the dual Newton's formula in \eqref{eq:u2} with the approximation in \eqref{eq:hat_lam}, we obtain a Newton-type dual update  
\#\label{eq:dual_update_1}
\lambda^{k+1}=\lambda^k + \beta_2H^kW\tilde{x}^k. 
\#

Now, by combining \eqref{eq:xp1} and \eqref{eq:dual_update_1}, we obtain our Newton-type approximation of MM as follows. At each iteration $k$, we have
\#\label{eq:up_2nd}
\tilde{x}^{k+1} = \tilde{x}^k - (\tilde{H}^k)^{-1}\nabla_{\tilde{x}}\tilde{L}(\tilde{x}^k,\lambda^k), \quad\lambda^{k+1}=\lambda^k + \beta_2H^kW\tilde{x}^k, 
\#
where $H^k = \nabla^2 f(x^k) + \mu I_{nd}$ and $\tilde{H}^k = \diag \{\mu nI_d, H^k\}$.

\subsection{FedHybrid to Handle System Heterogeneity}
Since different clients in the network have different computation capacities, we consider combining gradient-type method in \eqref{eq:up_1st} and Newton-type method in \eqref{eq:up_2nd} to provide a hybrid update framework. Specifically, all clients in the network can choose to perform gradient-type or Newton-type updates based on their computation capacities. For notational convenience, we denote by $J_1 = \{i\in [n] \colon \text{client $i$ performs gradient-type updates}\}$ and $J_2 = \{i\in [n] \colon \text{client $i$ performs Newton-type updates}\}$. 
Based on such choices of different update types, we propose our hybrid updates as follow. At each iteration $k$, we have
\#\label{eq:up}
\tilde{x}^{k+1} = \tilde{x}^k - \tilde{A}(\tilde{D}^k)^{-1}\nabla_{\tilde{x}}\tilde{L}(\tilde{x}^k,\lambda^k), \quad\lambda^{k+1}=\lambda^k + BD^kW\tilde{x}^k,
\#
where stepsize matrices $\tilde{A} = \diag\{a_0, a_1, \cdots, a_n\}\otimes I_d \in\RR^{(n+1)d\times(n+1)d}$ and $B = \diag\{b_1,\cdots,b_n\}\otimes I_d\in\RR^{nd\times nd}$ with personalized stepsizes $a_i, b_i>0$, and update matrices $\tilde{D}^k = \diag\{I_d, D^k\}\in\RR^{(n+1)d\times(n+1)d}$, and $D^k = \diag\{D_1^k,\cdots,D_n^k\}\in\RR^{nd\times nd}$. Here $D_i^k = I$ if $i\in J_1$ while $D_i^k = \nabla^2 f_i(x_i) + \mu I$ if $i\in J_2$. 
We remark that the updates in \eqref{eq:up} generalize the updates in both \eqref{eq:up_1st} and \eqref{eq:up_2nd}. 
On the one extreme, if $J_1 = [n]$ and $J_2 = \emptyset$, the updates in \eqref{eq:up} recovers the gradient-type updates in \eqref{eq:up_1st}; on the other extreme, if $J_1 = \emptyset$ and $J_2 = [n]$, the updates in \eqref{eq:up} recovers the Newton-type updates in \eqref{eq:up_2nd}. 
Based on \eqref{eq:up}, we propose our FedHybrid method with a distributed scheme in Algorithm \ref{algo:fedh}.

% Formally, we use block diagonal matrices $\tilde{D}$ and $D$ in replace of identity matrices in \eqref{eq:up_1st} and $\tilde{H}$ and $H$ in \eqref{eq:up_2nd}, respectively, where $\tilde{D}^k = \diag\{I_d, D^k\}\in\RR^{(n+1)d\times(n+1)d}$, $D^k = \diag\{D_1^k,\cdots,D_n^k\}\in\RR^{nd\times nd}$ with $D_i^k = I$ for clients $i\in J_1$ performing gradient-type updates and $D_i^k = \nabla^2 f_i(x_i) + \mu I$ for clients $i\in J_2$ performing Newton-type updates. 
% Thus, we obtain a hybrid method as follows. At each iteration $k$, 
% \#\label{eq:up}
% \tilde{x}^{k+1} = \tilde{x}^k - \tilde{A}(\tilde{D}^k)^{-1}\nabla_{\tilde{x}}\tilde{L}(\tilde{x}^k,\lambda^k), \quad\lambda^{k+1}=\lambda^k + BD^kW\tilde{x}^k,
% \#
% where stepsizes matrices $\tilde{A} = \diag\{a_0, a_1, \cdots, a_n\}\otimes I_d\in\RR^{(n+1)d\times(n+1)d}$ and $B = \diag\{b_1,\cdots,b_n\}\in\RR^{nd\times nd}$ with $a_i>0$ and $b_i>0$ denoting the primal and dual stepsizes used for client $i$. 

\begin{algorithm}[h] 
    \caption{FedHybrid: Hybrid Primal-dual Algorithm for distributed consensus optimization}
    \label{algo:fedh}
    \begin{algorithmic}[1]
    \State{\textbf{Input:} Initialization $x_0^0$, $x_i^0$, $\lambda_i^0 \in\RR^d$ for all $i\in[n]$, index sets $J_1$ and $J_2$, penalty parameter $\mu$, and stepsizes $a_i,b_i\in \RR^+$ for all $i\in [n]$.}  
    % \State{{Each client chooses its update type. $J_1, J_2$ are constructed.}}
    % \State{Choose constant $\mu>0$ and stepsizes $a_i$, $b_i$ for all $i\in[n]$} \algorithmiccomment{See Theorem \ref{thm:strong}, \ref{thm:general} for parameter choices}
    \For{$k =  1, \ldots, K-1$} 
    \State{Server sends $x_0^k$ to all clients in the network;}
        \For{each client $i\in J_1$ \textbf{parallel}}\algorithmiccomment{Gradient-type updates} \label{line:g_start}
           \State{$x_i^{k+1} = x_i^k - a_i(\nabla f_i(x_i^k) - \lambda_i^k + \mu(x_i^k - x_0^k))$;}
           \State{$\lambda_i^{k+1} = \lambda_i^k + b_i(x_0^k - x_i^k)$;}
           \State{Send $x_i^{k+1}$ and $\lambda_i^{k+1}$ to the server;}
        \EndFor \label{line:g_end}
        \For{each client $i\in J_2$ \textbf{parallel}}\algorithmiccomment{Newton-type updates} \label{line:n_start}
           \State{$x_i^{k+1} = x_i^k - a_i(\nabla^2 f_i(x_i) + \mu I_d)^{-1}(\nabla f_i(x_i^k) - \lambda_i^k + \mu(x_i^k - x_0^k))$;}
           \State{$\lambda_i^{k+1} = \lambda_i^k + b_i(\nabla^2 f_i(x_i) + \mu I_d)(x_0^k - x_i^k)$;}
           \State{Send $x_i^{k+1}$ and $\lambda_i^{k+1}$ to the server;}
        \EndFor \label{line:n_end}
        \State{Server updates: $x_0^{k+1} = \sum_{i\in[n]}x_i^{k+1}/n- \sum_{i\in[n]}\lambda_i^{k+1}/(\mu n)$.} \algorithmiccomment{Consensus update} \label{line:con}
    \EndFor
\end{algorithmic}
\end{algorithm}

Our FedHybrid in Algorithm \ref{algo:fedh} consists of three steps: gradient-type updates, Newton-type updates, and consensus update. Specifically, gradient-type (Lines \ref{line:g_start}--\ref{line:g_end}) and Newton-type updates (Lines \ref{line:n_start}--\ref{line:n_end}) follow from \eqref{eq:up} by extracting the corresponding block.  In the consensus update (Line \ref{line:con}), we choose the stepsize $a_0 = 1/(\mu n)$ in $\tilde A$ of \eqref{eq:up} and replace the primal and dual decision variables $x^k$ and $\lambda^k$ by their updated counterpart $x^{k+1}$ and $\lambda^{k+1}$. Then, by substitution, we obtain that 
\#\label{eq:x0}
x_0^{k+1} = \frac{1}{n}\sum_{i\in[n]}x_i^{k+1} - \frac{1}{\mu n}\sum_{i\in[n]}\lambda_i^{k+1}, 
\#
which corresponds to the consensus update in Line \ref{line:con}.

% When choosing $a_0 = 1/(\mu n)$, we can cancel out $x_0^k$ at the central server and the obtain consensus update in the form of $x_0^{k+1} = \sum_{i\in[n]}x_i^{k}/n - \sum_{i\in[n]}\lambda_i^{k}/{(\mu n)}.$ If putting the consensus update at the end of each iteration and using the updated primal and dual decision variables, we obtain the consensus update as follows,
% \#\label{eq:x0}
% x_0^{k+1} = \frac{1}{n}\sum_{i\in[n]}x_i^{k+1} - \frac{1}{\mu n}\sum_{i\in[n]}\lambda_i^{k+1}.
% \#
% Thus, we obtain our FedHybrid method using primal-dual updates in \eqref{eq:up} and consensus update in \eqref{eq:x0}. The FedHybrid method with a distributed scheme is given in Algorithm \ref{algo:fedh}.

\section{Convergence Analysis}

This section presents the convergence results for FedHybrid in Algorithm \ref{algo:fedh}, regardless of clients' choices of gradient-type or Newton-type updates. In Section \ref{sec:s_conv}, we show that FedHybrid converges to the optimal solution at a linear rate if the local objective $f_i(\cdot)$ is strongly convex for all $i\in [n]$. In Section \ref{sect:proof}, we provide the proof sketch of the theorem.
% provide a convergence result for general convex objective function.
Before presenting the results, we first introduce our performance metric and a reformulation of the augmented Lagrangian function $\tilde{L}$.

\noindent\textbf{Performance Metric.} 
% Since FedHybrid method tackles the primal and dual problems simultaneously using single-loop updates, 
We define the primal tracking error and the dual optimality gap  as follows, 
\#\label{eq:terrors}
\Delta_{\tilde{x}}^k = \tilde{L}(\tilde{x}^k,\lambda^k) - \tilde{L}(\tilde{x}^*(\lambda^k),\lambda^k), \quad \Delta_{\lambda}^k = g(\lambda^*) - g(\lambda^k),
\#
where $\lambda^*$ is the optimal solution to the dual problem in \eqref{prob:dual}. Here $\Delta_{\tilde{x}}^k$ quantifies how close the augmented Lagrangian function at $\tilde{x}^k$ is from the optimal value of the inner problem given $\lambda^k$. In Section \ref{sect:proof}, we will present a novel proof framework to show the convergence of both $\Delta_{\lambda}^k$ and $\Delta_{\tilde{x}}^k$. We remark that under Assumption \ref{ass:hessian}, the convergence of $\Delta_{\lambda}^k$ ensures that the dual variable sequence $\{\lambda^k\}_{k \in [K]}$ converges to the optimal solution $\lambda^*$ and the convergence of $\Delta_{\tilde{x}}^k$ ensures that the primal variable sequence $\{\tilde{x}^k\}_{k \in [K]}$ converges to $\tilde{x}^*(\lambda^*)$, where $\tilde{x}^*(\lambda^*)$ is the optimal solution of the original problem in \eqref{prob:constrained} due to the strong duality.

\noindent\textbf{Reformulation of $\tilde{L}$ Based on Consensus Update.} 
With slight abuse of notations, the consensus update \eqref{eq:x0} in Algorithm \ref{algo:fedh} can be written as $x_0^{k} = x_0(x^{k},\lambda^{k})$, where $x_0:\RR^{nd}\times\RR^{nd}\to\RR$ such that $x_0(x, \lambda) = \rbr{\mathbbm{1}_n^\intercal \otimes I_d} \rbr{x/n-\lambda/(\mu n)}$ for any $x, \lambda \in \RR^{nd}$. 
By substituting $x_0 = x_0(x, \lambda)$ in the augmented Lagrangian function $\tilde L$ defined in \eqref{eq:lagrangian_func}, we have $\tilde{L}(\tilde{x},\lambda) = \tilde{L}(x_0(x, \lambda),x,\lambda)$. 
Motivated by this, we define $L:\RR^{nd}\times\RR^{nd}\to \RR$ such that $L(x,\lambda) = \tilde{L}(x_0(x, \lambda),x,\lambda)$, which can be shown to be equivalent. See Section \ref{sect:a2} for details. For convenience, we use $L$ in the convergence analysis.
%In Section \ref{subsect:equiv_l_tL}, we show that using functions $L$ or $\tilde{L}$ in the analysis are indeed equivalent. 
%We remark that $\Delta_{\tilde{x}}^k = L(x^k, \lambda^k) - L(x^*(\lambda^k), \lambda^k)$ by its definition in \eqref{eq:terrors}. For convenience, we use the function $L$ in replace of $\tilde{L}$ hereafter.

\subsection{Convergence of FedHybrid for Strongly Convex Function}\label{sec:s_conv}

 In this section, we show that FedHybrid converges to the optimal solution at a linear rate under Assumption \ref{ass:hessian}. The following theorem states the results.

% The following theorem proves the linear convergence rate of FedHybrid method for strongly convex function.
\begin{theorem}\label{thm:strong}
For any given $\mu > 0$, under Assumption \ref{ass:hessian}, we suppose that the stepsizes $\{a_i, b_i\}_{i\in [n]}$ satisfy the following condition,
\#
&a_i \le 1/(22\mu/9 + 2\ell), \quad b_i\le\min\{\mu/9, \underline{\alpha}m^2/21\}, \quad  \text{for } i\in J_1, \notag \\
&a_i \le (m_i + \mu)/(22\mu/9 + 2\ell), \quad  b_i\le\min\{\mu/9, \underline{\alpha}m^2/21\}/(\ell_i + \mu), \quad  \text{for } i\in J_2, \label{eq:steps}
\#
where $\underline{\alpha} = \min\{\min_{i\in J_1}\{a_i\}, \min_{i\in J_2}\{a_i/(\ell_i + \mu)\}\}$. Then for all $k = 0,1,\cdots, K-1$, the iterates from FedHybrid satisfy 
\$
\Delta^{k+1} \le \rbr{1-\rho}\Delta^k,
\$
where $\rho = \min\left\{3\underline{\beta}/{(13m+13\mu)}, {m\underline{\alpha}}/{2}\right\}$ with $\underline{\beta} = \min\{\min_{i\in J_1}\{b_i\}, \min_{i\in J_2}\{b_i(m_i+\mu)\}\}$, and $\Delta^k = 13\Delta_{\lambda}^{k} + \Delta_{\tilde x}^{k}$. Here $\Delta_{\lambda}^{k}$ and $\Delta_{\tilde x}^{k}$ are defined in \eqref{eq:terrors}. 
\end{theorem}

\begin{remark}
Theorem \ref{thm:strong} provides a last iterate linear (Q-linear) rate of convergence of the error term $\Delta^k$, which combines the dual optimality gap $\Delta_\lambda^k$ and the primal tracking error $\Delta_{\tilde{x}}^k$. This ensures that FedHybrid converges linearly to the optimal solution, regardless of clients' choices of gradient-type or Newton-type updates. Also, in FedHybrid, each client can choose not only the type of updates based on their computation capacities, but also personalized stepsizes related to the properties of their local objective functions, which provides flexibility to handle heterogeneity in practice.
\end{remark}
% In FedHybrid method, each client can choose not only the type of updates based on their computation capacities, but also personalized stepsizes related to the properties of their local objective functions.Furthermore, although the worst-case convergence rate is linear in analysis, which is bottlenecked by those clients performing gradient-type updates, as shown in numerical experiments in Section \ref{sect:experiments}, the method can achieve a much faster convergence rate with more clients performing Newton-type updates. Note that the rate $1-\rho$ here is related to the strong convexity parameter $m$, where a slightly larger $m$ leads to a faster convergence rate.

\subsection{Proof Sketch} \label{sect:proof}
In this section, we provide a novel proof framework for analyzing primal-dual algorithms.  We apply such analysis to FedHybrid and obtain the convergence results in above theorems.

% While here we provide tighter bounds on the constants $m_g$ and $\ell_g$ utilizing the special structure of the server-client network.

Due to the coupled nature of primal and dual updates in the algorithm, our main idea for analyzing primal-dual methods is to upper bound the dual optimality gap $\Delta_{\lambda}^{k}$ and the primal tracking error $\Delta_{\tilde{x}}^k$ through coupled inequalities. We remark that the dual function $g(\cdot)$ is $m_g$-strongly concave with $\ell_g$-Lipschitz continuous gradient and the primal function $L(\cdot,\lambda)$ is $m$-strongly concave with $\ell_L$-Lipschitz continuous partial gradient $\nabla_x L(x,\lambda)$ for any fixed $\lambda$. See Section \ref{sect:a3} for details.
We decompose our proof framework into the following three steps.

\noindent\textbf{Step 1.} We derive a bound of dual optimality gap $\Delta_{\lambda}^{k+1}$ with an alternative primal tracking error $\norm{\nabla_xL({x}^k,\lambda^k)}^2$ in the following lemma. This quantifies how close $\lambda^{k}$ is from the dual optimal $\lambda^*$.

\begin{lemma} \label{lem:gL}
Under Assumption \ref{ass:hessian}, given a constant $\mu>0$, we suppose that the stepsizes satisfy \eqref{eq:steps}. It holds for all $k = 0,1,\cdots,K-1$ that 
\$
\Delta_{\lambda}^{k+1} \le \Delta_{\lambda}^{k} - \rbr{\frac{1}{2}-\beta\ell_g}\norm{\nabla g(\lambda^{k})}^2_{BD^k} + \rbr{\frac{1}{2}+\beta\ell_g}\frac{\beta}{m^2} \norm{\nabla_xL({x}^k,\lambda^k)}^2,
\$
where $\beta = \max\{\max_{i\in J_1}\{b_i\}, \max_{i\in J_2}\{b_i(\ell_i+\mu)\}\}$ and $\ell_g$ is the Lipschitz constant of $g(\cdot)$.
\end{lemma}

We remark that the derivation of Lemma \ref{lem:gL} relies on the Lipschitz gradient property of the dual function $g(\cdot)$. See Section \ref{sect:a4} for a detailed proof of Lemma \ref{lem:gL}.

\noindent\textbf{Step 2.} We provide an upper bound of $\Delta_{\tilde{x}}^{k+1}$ involving the dual optimality gap $\Delta_{\lambda}^k$ as follows. We remark that the primal tracking error at iteration $k+1$ consists of the following three parts.
\$
\Delta_{\tilde{x}}^{k+1} = \underbrace{L(x^{k+1}, \lambda^{k+1}) -L(x^{k+1}, \lambda^{k})}_{\textstyle {\text{increase due to dual update}}} & + \underbrace{L(x^{k+1}, \lambda^{k}) -L(x^*(\lambda^{k}), \lambda^{k})}_{\textstyle {\text{updated primal tracking error}}} \\
& \qquad \qquad + \underbrace{L(x^*(\lambda^{k}), \lambda^{k}) - L(x^*(\lambda^{k+1}), \lambda^{k+1})}_{\textstyle {\text{difference between dual optimality gaps}}}. \notag
\$
We can upper bound the updated primal tracking error by $\Delta_{\tilde{x}}^k$ using the Lipschitz property of $\nabla_x L(x,\lambda)$ with respect to $x$ and upper bound the other two parts using dual information. As a result, we obtain the following lemma, whose proof is deferred to Section \ref{sect:a4}.
\begin{lemma} \label{lem:Lg}
Under Assumption \ref{ass:hessian}, given a constant $\mu>0$, we suppose that the stepsizes satisfy \eqref{eq:steps}. It holds for all $k = 0,1,\cdots,K-1$ that  
\$
\Delta_{\tilde{x}}^{k+1} \le&  \Delta_{\tilde{x}}^k  + \kappa\norm{\nabla g(\lambda^k)}^2_{BD^k} - \norm{\nabla_x L(x^{k}, \lambda^{k})}_{P^k}^2+ \Delta_{\lambda}^k - \Delta_{\lambda}^{k+1}, 
\$
where { $\kappa = 3 + {2\beta^2}/{\mu^2}+{\beta}/{\mu}$} and {$P^k = A(D^k)^{-1} - (\beta+{\ell_L}/{2})A^2(D^k)^{-2}-\beta \kappa / {m^2}I$}, $\beta$ is defined in Lemma \ref{lem:gL} and $\ell_L$ is the Lipschitz constant of $\nabla_x L(x,\lambda)$ with respect to $x$.
\end{lemma}
% See Section \ref{sect:a} for a detailed proof of Lemma \ref{lem:Lg}. 

\noindent\textbf{Step 3.} We combine the above coupled inequalities to provide convergence results in Theorem \ref{thm:strong}. Using a weighted sum of the two inequalities in  Lemmas \ref{lem:gL} and \ref{lem:Lg}, and combining the strong-convexity of $L(\cdot, \lambda)$ and the strong-concavity of $g(\cdot)$, we conclude the proof of Theorem \ref{thm:strong}. See Section \ref{sect:b} for a detailed proof of the theorems.

%!TEX root = main.tex

\section{Numerical Experiments}\label{sect:experiments}
In this section, we present experimental results for FedHybrid on some convex distributed optimization problems. In particular, we consider least squares problems and binary classification problems in a server-client network, both over synthetic and real-life datasets. 

\noindent\textbf{Experimental Setup.}
We evaluate all algorithms on four setups with non-IID data partitioning, where in each setup, there are $n$ clients with $d$-dimensional decision variables: (1) Linear regression on a synthetic dataset: $n=20$, $d=30$. (2) Linear regression on non-IID partitioned Boston housing prices dataset: $n=8$, $d=14$. (3) Regularized logistic regression on a synthetic dataset: $n=20$, $d=30$. (4) Regularized logistic regression on non-IID partitioned mushroom dataset: $n=8$, $d=99$. 
For synthetic datasets, we generate data following different distributions across each client.  For real datasets, we distribute the data to each client according to the values of the responses.
% sort the responses in an ascending order and distribute data to each client by turns.
In all setups, the local dataset sizes are randomly sampled. See Section \ref{sect:c1} for a detailed description of the setup. 
% We defer the details of setups in Section \ref{sect:experiment_details}.

\noindent\textbf{Compared Methods.}
We focus on comparing the following algorithms designed for the server-client network: Federated Averaging (FedAvg) \citep{mcmahan2017communication}, Distributed Self-Concordant Optimization (DiSCO) \citep{zhang2015disco}, and FedHybrid with number of $K$ clients performing Newton-type updates while all others doing gradient-type updates (FedH-$K$). We remark that for FedH-$K$, we have $|J_2| = K$. In particular, FedH-G and FedH-N are short for FedHybrid with all clients performing gradient-type and Newton-type updates, respectively. 

We remark that FedAvg is the baseline of first-order method for the server-client network. For a fair comparison, we consider the non-stochastic version of FedAvg, that is, at each iteration, all agents in the network take a full gradient descent step. Among second-order methods, DiSCO is an inexact damped Newton method that performs well in practice. We choose DiSCO instead of DANE \citep{shamir2014communication}, AIDE \citep{reddi2016aide} or GIANT \citep{wang2018giant} as a baseline second-order method since the other methods require the assumption that each client has access to IID sampled data points, which is generally not realistic in practice. 
For each method, we tune parameters using grid search in the range $[10^{-4}, 10]$ and stepsizes in the range $[10^{-4}, 1]$ and choose the best one that minimizes the optimality gap. 

\noindent\textbf{Experimental Results.}
As shown in Figure \ref{fig:compare_1}, with some clients in the network performing Newton-type updates, FedHybrid improves the overall training speed a lot and outperforms the baseline method FedAvg consistently. See Section \ref{sect:c2} for more results.

\begin{figure}[h]
     \begin{subfigure}[b]{0.5\textwidth}
         \centering
         \includegraphics[width=\textwidth]{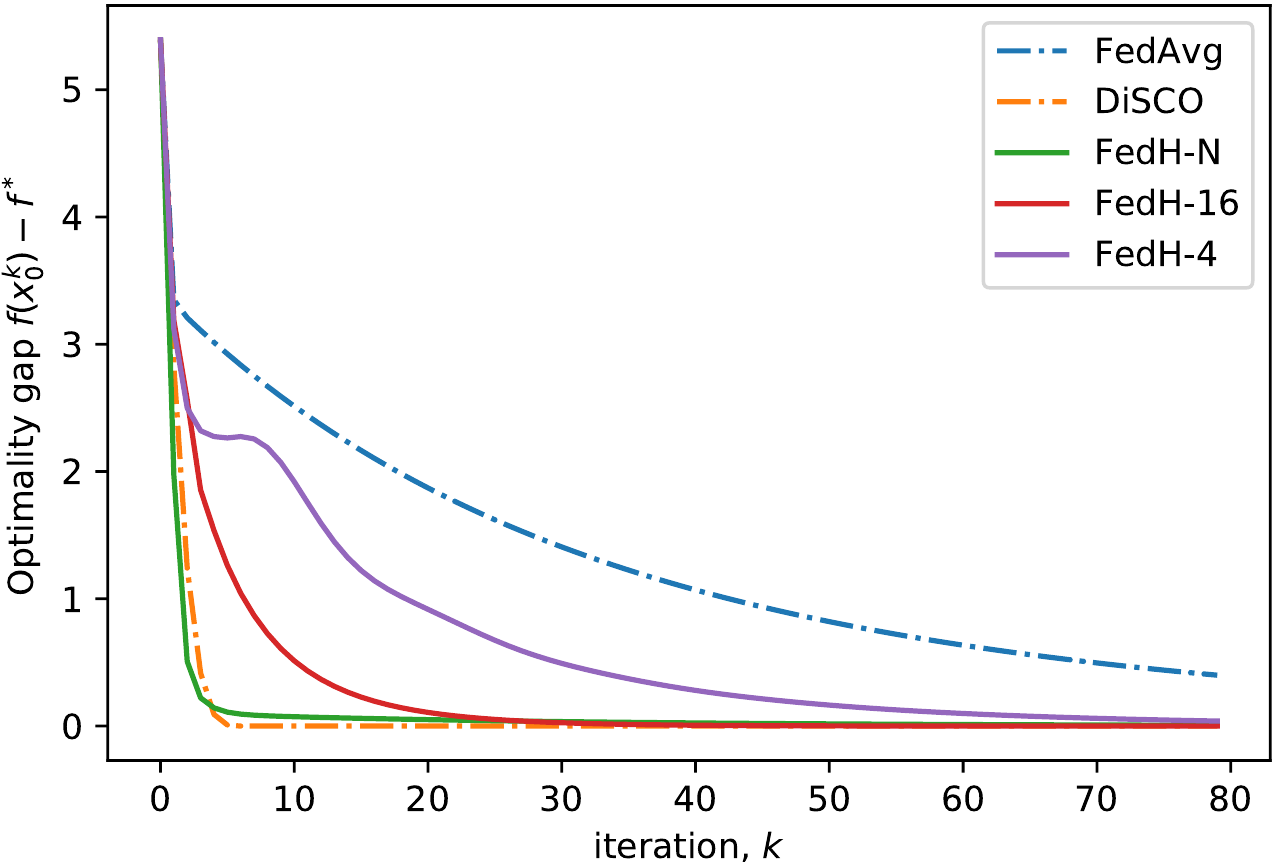}
         \caption{\small Least squares on synthetic dataset (setup (1)).}
     \end{subfigure}
     \begin{subfigure}[b]{0.5\textwidth}
         \centering
         \includegraphics[width=\textwidth]{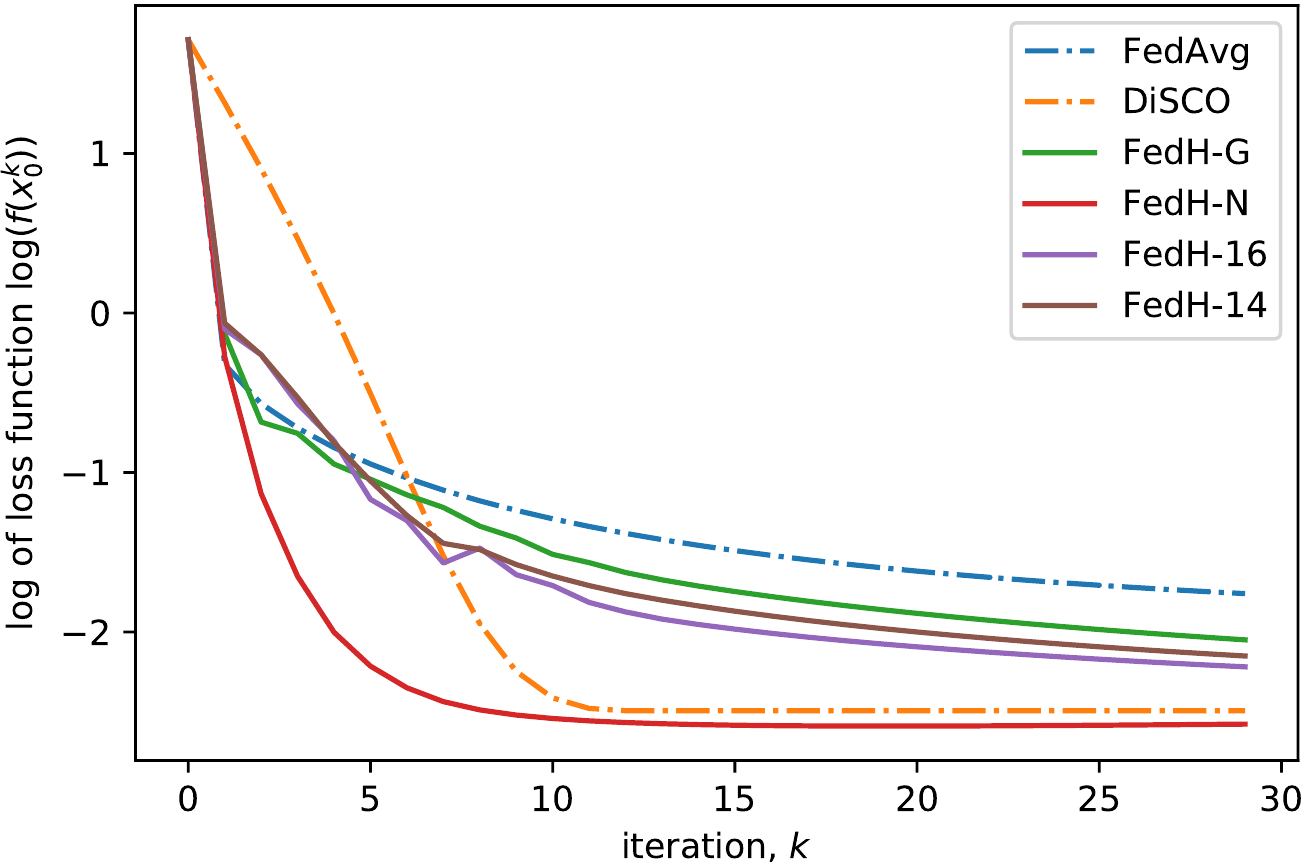}
         \caption{{Logistic regression on Mushroom dataset (setup (4)).}}
     \end{subfigure}
     \caption{ Comparison of convergence of different algorithms}
     \label{fig:compare_1}
\end{figure}

In traditional distributed optimization algorithms, all clients are performing the same updates. The complexity of the method is determined by the client equipped with the worst computation hardware. While in our proposed FedHybrid framework, since for some parts of the clients, efficient Newton-type updates are involved, the overall system enjoys a faster convergence speed compared to systems running gradient-type methods only. Thus we can maximally leverage the parallel heterogeneous computation capabilities in this setting.

In particular, if all clients in the network perform Newton-type updates, our second-order FedH-N method achieves a comparable convergence performance or outperforms DiSCO. The reason is that DiSCO uses a preconditioning matrix, which performs the best only when the clients' data are IID.

\noindent\textbf{Relation between convergence rate and the number of Newton-type clients.}
As shown in Figure \ref{fig:compare_2}, as the number of clients that perform Newton-type updates increases, the convergence rate of FedHybrid becomes faster. This suggests that those clients with higher computational capabilities and/or cheaper cost to perform computation can choose to implement Newton-type updates locally to help speedup the overall training speed of the system.

\begin{figure}[h]
     \begin{subfigure}[b]{0.5\textwidth}
         \centering
         \includegraphics[width=\textwidth]{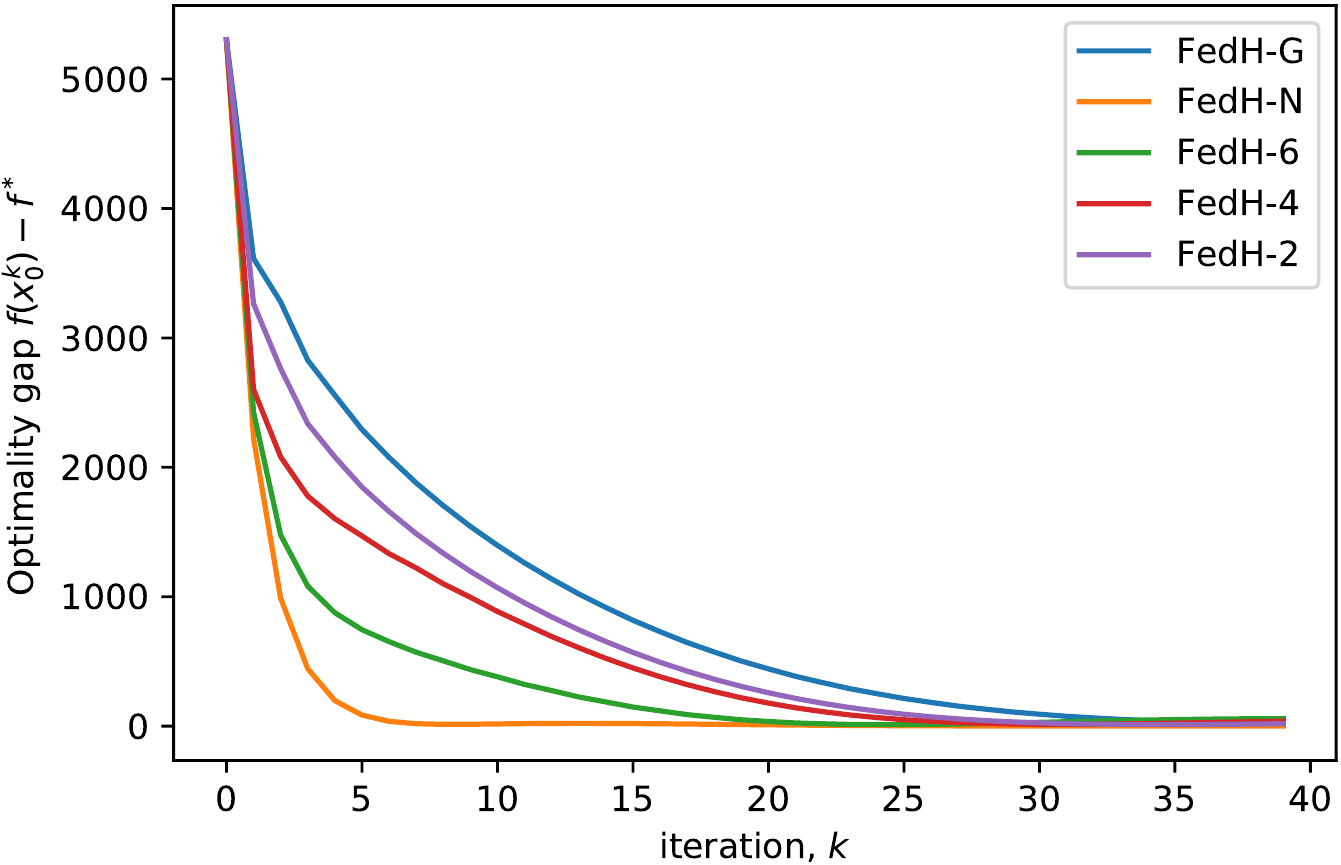}
         \caption{\small Least squares on Boston dataset (setup (2)).}
     \end{subfigure}
     \begin{subfigure}[b]{0.5\textwidth}
         \centering
         \includegraphics[width=\textwidth]{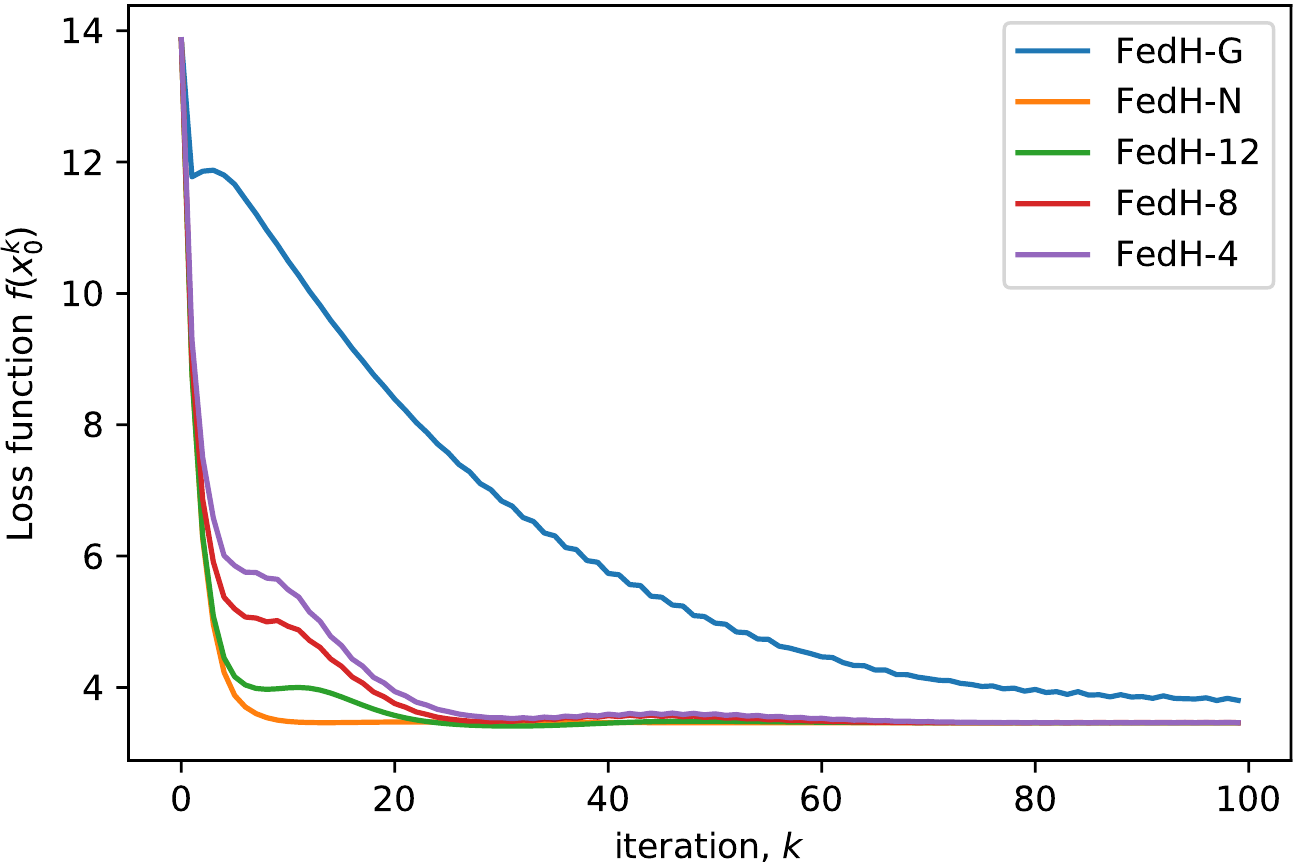}
         \caption{\small Logistic regression on synthetic dataset (setup (3)).}
     \end{subfigure}
     \caption{Comparison of convergence rate with different number of Newton-type clients.}
     \label{fig:compare_2}
\end{figure}

\section{Conclusion}

This paper proposes FedHybrid, a distributed hybrid primal-dual algorithm framework that allows clients to perform either gradient-type or Newton-type updates based on their computation capacities. We provide a novel analysis framework for primal-dual algorithms and obtain a linear convergence result of FedHybrid for strongly convex objective functions. Numerical studies are also provided to demonstrate the efficacy of FedHybrid in practice. 
We remark that while experiments show that FedHybrid can achieve a better performance with more clients performing Newton-type updates in the network, our theoretical convergence bounds are independent of the number of clients performing Newton-type updates.

We highlight a few interesting directions for future work on FedHybrid and federated optimization. First, we could consider stochastic FedHybrid methods. For instance, each client in the network could perform stochastic gradient-type or subsampled Newton-type methods on its local dataset. Moreover, in federated applications, at each communication round, it's possible that only a small subset of the clients are active. Thus, we could consider asynchronous updates in FedHybrid, where only a randomly selected subset of the clients perform updates at each iteration. Also, we expect FedHybrid could be generalized to broader settings, such as time-varying graphs and/or systems with non-convex objective functions. 

%See Section \ref{sect:d} for further discussion of limitations and future works.

\newpage
\bibliographystyle{abbrvnat}
\bibliography{cite}

\newpage
\appendix 
%!TEX root = main.tex
\section{Proof of Lemmas} \label{sect:a}
In this section, we provide some lemmas with analysis. These results will help us to build a better understanding of the problem and will be used to prove the convergence theorems in the next section.

\subsection{Proof of Lemma \ref{lem:dual_hessian_app}} \label{sect:a1}
This section provides the proof of Lemma \ref{lem:dual_hessian_app}.
\begin{proof}
Following from the approximated dual Newton update formula $\hat \nabla^2 g(\lambda^k) \Delta \breve\lambda^k = \hat \nabla g(\lambda^k)$ and the explicit form of the dual gradient $\hat\nabla g(\lambda^k)$ and the dual Hessian $\hat\nabla^2 g(\lambda^k)$ given in \eqref{eq:app_na}, we have
\$
-W\rbr{\nabla_{\tilde{x}\tilde{x}}^2\tilde L(\tilde{x}^k,\lambda^k)}^{-1}W^\intercal \Delta\breve\lambda^k =  W\tilde{x}^k.
\$
We note that the null space of matrix $W = \rbr{\mathbbm{1}_n, -I_n}\otimes I_d$ is $\text{Null}(W) = \{\mathbbm{1}_{n+1}\otimes y: y\in\RR^d\}$. Thus, there exists $y^k\in\RR^d$ such that
\#\label{eq:wlambda1}
\rbr{\nabla_{\tilde{x}\tilde{x}}^2\tilde L(\tilde{x}^k,\lambda^k)}^{-1}W^\intercal \Delta\breve\lambda^k + \tilde{x}^k = \mathbbm{1}_{n+1}\otimes y^k.
\#
Rearranging terms in \ref{eq:wlambda1}, we have 
\#\label{eq:wlambda2}
W^\intercal \Delta\breve\lambda^k & = \nabla_{\tilde{x}\tilde{x}}^2\tilde L(\tilde{x}^k,\lambda^k)\rbr{\mathbbm{1}_{n+1}\otimes y^k - \tilde{x}^k} \notag \\
& = \rbr{\nabla^2 \tilde{f}(\tilde{x}^k) + \mu W^\intercal W}\rbr{\mathbbm{1}_{n+1}\otimes y^k - \tilde{x}^k},
\#
where the last equality follows by substituting the definition of $\nabla_{\tilde{x}\tilde{x}}^2\tilde L(\tilde{x}^k,\lambda^k)$. Since $(\mathbbm{1}_{n+1}^\intercal\otimes I_d)W^\intercal = 0$, multiplying $\mathbbm{1}_{n+1}^\intercal\otimes I_d$ on both sides of \ref{eq:wlambda2}, we have
\$
0 = \nabla^2 \tilde{f}(\tilde{x}^k)\rbr{\mathbbm{1}_{n+1}\otimes y^k - \tilde{x}^k}.
\$
Thus, we have $(\sum_{i\in[n]} \nabla^2 f_i(x_i^k))y^k =\sum_{i\in[n]} \nabla^2 f_i(x_i^k)x_i^k$. Therefore, we have
\#\label{eq:y}
y^k = \text{\large$($}\sum_{i\in[n]} \nabla^2 f_i(x_i^k)\text{\large$)$}^{-1}\sum_{i\in[n]} \nabla^2 f_i(x_i^k)x_i^k.
\#
Substituting \ref{eq:y} into \ref{eq:wlambda2}, we have 
\#
W^\intercal \Delta\breve\lambda^k = \nabla_{\tilde{x}\tilde{x}}^2\tilde L(\tilde{x}^k,\lambda^k)\rbr{\mathbbm{1}_{n+1}\otimes y^k - \tilde{x}^k}, 
\#
where $y^k = (\sum_{i\in[n]} \nabla^2 f_i(x_i^k))^{-1}\sum_{i\in[n]} \nabla^2 f_i(x_i^k)x_i^k$. This concludes the proof of Lemma \ref{lem:dual_hessian_app}.
\end{proof}

\subsection{Reformulation Based on Consensus Update} \label{sect:a2}
In this section, we reformulate the augmented Lagrangian function $\tilde{L}$ defined in \eqref{eq:lagrangian_func} based on the consensus update \eqref{eq:x0} in Algorithm \ref{algo:fedh} and obtain a function $L$. We show equivalence of functions $\tilde{L}$ and $L$ used in the analysis of FedHybrid. For convenience, we will use $L$ in the subsequent analysis.

With slight abuse of notations, we note that the consensus update \eqref{eq:x0} in Algorithm \ref{algo:fedh} can be written as $x_0^{k} = x_0(x^{k},\lambda^{k})$, where $x_0:\RR^{nd}\times\RR^{nd}\to\RR$ such that for any $x, \lambda \in \RR^{nd}$, 
\#\label{eq:x0_func}
x_0(x, \lambda) = \rbr{\mathbbm{1}_n^\intercal \otimes I_d} \rbr{x/n-\lambda/(\mu n)}.
\# 
By substituting $x_0 = x_0(x, \lambda)$ in the consensus update defined in \ref{prob:constrained}, we have 
\#\label{eq:a1}
W \tilde{x}=\mathbbm{1}_n\otimes x_0 (x, \lambda)- x= -M x -Z \lambda/\mu,
\#
where matrices $Z = \mathbbm{1}_n\mathbbm{1}_n^\intercal\otimes I_d/n\in\RR^{nd\times nd}$ and $M=I_{nd}-Z\in\RR^{nd\times nd}$. It's easy to show that $Z^2 = Z$ and $M^2 = M$, which implies that $0\preceq Z\preceq I_{nd}$ and $0\preceq M\preceq I_{nd}$. Moreover, by substituting $x_0 = x_0(x, \lambda)$ in the augmented Lagrangian function $\tilde L$ defined in \eqref{eq:lagrangian_func}, we have $\tilde{L}(\tilde{x},\lambda) = \tilde{L}(x_0(x, \lambda),x,\lambda)$. 
Motivated by this, we define $L:\RR^{nd}\times\RR^{nd}\to \RR$ such that $L(x,\lambda) = \tilde{L}(x_0(x, \lambda),x,\lambda)$, we have 
\#\label{eq:a2}
L(x,\lambda) = f(x) - \lambda^\intercal Mx + \frac{\mu}{2}x^\intercal Mx - \frac{1}{2\mu}\lambda^\intercal Z\lambda.
\#
We provide the following two lemmas to show some pivotal properties with the reformulation formula defined in \eqref{eq:x0_func} to be used in the analysis of FedHybrid. The following lemma shows the equivalence of partial gradients of functions $\tilde L$ and $L$ with respect to $x$.
\begin{lemma} \label{lem:L}
Under Assumption \ref{ass:hessian}, consider the augmented Lagrangian function $\tilde{L}(x_0, x, \lambda)$ defined in \eqref{eq:lagrangian_func} and the function $L(x,\lambda)$ defined in \eqref{eq:a2}, with the consensus update $x_0 = x_0(x,\lambda)$, it holds that, for any $x, \lambda\in\RR^{nd}$,
\$
\nabla_xL(x,\lambda) = \nabla_x\tilde{L}(x_0(x,\lambda), x,\lambda).
\$
\end{lemma}

\begin{proof}
Taking partial gradient with respect to $x_0$ in \eqref{eq:lagrangian_func}, we have
\#\label{eq:nx0}
\nabla_{x_0}\tilde{L}(x_0,x,\lambda) &= \rbr{\mathbbm{1}_n\otimes I_d}^\intercal\rbr{\lambda + \mu\rbr{\mathbbm{1}_n\otimes x_0 - x}}\notag\\
& = \mu n\sbr{x_0-x_0(x,\lambda)} \notag\\
& = 0,
\#
where the second equality follows from the definition of $x_0(x,\lambda)$ and the last equality follows from the consensus update $x_0 = x_0(x,\lambda)$. Then taking partial gradient with respect to $x$ in \eqref{eq:lagrangian_func}, we have
\$
\nabla_x\tilde{L}(x_0,x,\lambda) &= \nabla f(x) - \lambda + \mu (x - \mathbbm{1}_n\otimes x_0) \\
&=\nabla f(x) - \lambda + \mu (x - \mathbbm{1}_n\otimes x_0(x,\lambda)) \\
& = \nabla f(x) - \lambda + \mu \rbr{x-\rbr{Zx-Z\lambda/\mu}} \\
& = \nabla f(x) - M\lambda+ \mu Mx,
\$
where the second equality follows from the consensus update $x_0 = x_0(x,\lambda)$ and the last equality follows from the definition of $x_0(x,\lambda)$ in \eqref{eq:x0_func}. By taking partial gradient with respect to $x$ in \eqref{eq:a2} and using the above equation, we have 
\#\label{eq:LtL}
\nabla_x L(x,\lambda) = \nabla f(x) - M\lambda+ \mu Mx = \nabla_x\tilde{L}(x_0,x,\lambda).
\#

Thus, by the chain rule in calculus, we have
\$
\nabla_x\tilde{L}(x_0(x,\lambda), x,\lambda) &= \nabla_xx_0(x,\lambda)\nabla_{x_0}\tilde{L}(x_0,x,\lambda) + \nabla_x\tilde{L}(x_0,x,\lambda) \\
&= \nabla_x\tilde{L}(x_0,x,\lambda)  \\
& = \nabla_x L(x,\lambda),
\$
where the second equality follows from \eqref{eq:nx0} and the last equality follows from \eqref{eq:LtL}. This concludes the proof of the lemma.
\end{proof}
We remark that using Lemma \ref{lem:L}, Algorithm \ref{algo:fedh} can be rewritten in the compact form as follows. At each iteration $k$, FedHybrid takes the following steps,
\#\label{algo:reformulate}
x^{k+1} &= x^k - A(D^k)^{-1}\nabla_x L(x^k,\lambda^k), \notag\\
\lambda^{k+1} &= \lambda^k + BD^kW\tilde{x}^k, \notag\\
x_0^{k+1} &= x_0(x^{k+1}, \lambda^{k+1}).
\#
The equivalent compact form in \eqref{algo:reformulate} of FedHybrid will be used in the subsequent analysis.

The following lemma shows a pivotal property of the primal optimal value $\tilde{x}^*(\lambda^k)$ for the inner problem in Problem \ref{prob:dual} with $\lambda = \lambda^k$.

\begin{lemma}\label{lem:x0star}
We consider the primal optimal value $\tilde{x}^*(\lambda) = ({x}_0^*(\lambda)^\intercal, x^*(\lambda)^\intercal)^\intercal$ for the inner problem in Problem \ref{prob:dual}. It holds for $\lambda\in\RR^{nd}$ that, 
\$
x_0^*(\lambda) = x_0(x^*(\lambda), \lambda),
\$ where the function $x_0$ is defined in \eqref{eq:x0_func}.
\end{lemma}
\begin{proof}
Using the first-order optimality condition of the inner problem in Problem \ref{prob:dual}, we have
\$
0 = \nabla_{\tilde{x}}\tilde{L}(\tilde{x}, \lambda)\given_{\tilde{x} = \tilde{x}^*(\lambda)} = \nabla\tilde{f}(\tilde{x}^*(\lambda)) + W^\intercal\lambda + \mu W^\intercal W\tilde{x}^*(\lambda).
\$
Consider the block corresponding to the central decision variable $x_0$ in the above equation, we have
\$
\rbr{\mathbbm{1}_n^\intercal\otimes I_d}\lambda + \mu n \tilde{x}_0^*(\lambda) - \mu \rbr{\mathbbm{1}_n^\intercal\otimes I_d}x^*(\lambda) = 0.
\$
Rearranging terms in the above equation, we have
\$
 \tilde{x}_0^*(\lambda) &= \rbr{\mathbbm{1}_n^\intercal\otimes I_d}\rbr{ x^*(\lambda)/n - \lambda/(\mu n)} \\
 &= x_0(x^*(\lambda), \lambda), 
\$
where the last equality follows from the definition of the function $x_0(x,\lambda)$ in \eqref{eq:x0_func}. This concludes the proof of the lemma. 
\end{proof}
We remark that $L(x^*(\lambda^k),\lambda^k) = \tilde{L}(x_0(x^*(\lambda^k),\lambda^k), x^*(\lambda^k), \lambda^k) = \tilde{L}(\tilde{x}^*(\lambda^k),\lambda^k)$. Then since  $L(x^k,\lambda^k) = \tilde{L}(x_0(x^k,\lambda^k), x^k, \lambda^k) = \tilde{L}(\tilde{x}^k,\lambda^k)$, let $\Delta_{x}^k = L(x^k,\lambda^k) - L(x^*(\lambda^k),\lambda^k)$, we have
\# \label{eq:dx_eq}
\Delta_{x}^k = \Delta_{\tilde x}^k,
\#
where $\Delta_{\tilde x}^k$ is the primal tracking error defined in \eqref{eq:terrors}. For convenience, we will use $\Delta_{x}^k$ in the subsequent analysis.

The following lemma highlights the optimality condition of the inner problem in Problem \ref{prob:dual}.
\begin{lemma} \label{lem:x_opt} We consider the primal optimal value corresponding to clients, $x^*(\lambda)$, for the inner problem in Problem \ref{prob:dual}. It holds for any $\lambda\in\RR^{nd}$ that, 
\$
\nabla_xL(x^*(\lambda),\lambda) = 0.
\$
\end{lemma}
\begin{proof}
Following from Lemma \ref{lem:L}, we have
\$
\nabla_xL(x^*(\lambda),\lambda) &= \nabla_xL\rbr{x_0(x^*(\lambda),\lambda), x^*(\lambda),\lambda} \notag \\
&= \nabla_xL\rbr{x_0^*(\lambda), x^*(\lambda),\lambda} \notag \\
& = \nabla_xL\rbr{\tilde{x}^*(\lambda),\lambda}\notag \\
& = 0,
\$
where the second equality follows from Lemma \ref{lem:x0star}, and the last equality follows from the first-order optimality condition of the inner problem in Problem \ref{prob:dual}.
\end{proof}

\subsection{Properties of the Dual and the Primal Functions} \label{sect:a3}
In this section, we provide lemmas showing some pivotal properties of the dual function $g(\cdot)$ and the primal function $L(\cdot,\lambda)$ for any fixed $\lambda\in\RR^{nd}$, respectively.

\begin{lemma}\label{lem:g}
Under Assumption \ref{ass:hessian}, the dual function $g(\cdot)$ defined in \eqref{prob:dual} is $m_g$-strongly concave with $m_g = {1}/{(\mu+\ell)}$ and the gradient of the dual function $\nabla g(\lambda)$ is $\ell_g$-Lipschitz continuous with $\ell_g ={1}/{\mu}$.
\end{lemma}
\begin{proof}
Following from Lemma \ref{lem:dual_hessian}, the Hessian of the dual function $g(\cdot)$ is given by:
\$
\nabla^2 g(\lambda) = -W\rbr{\nabla^2 \tilde{f}(\tilde{x}^*(\lambda)) + \mu W^\intercal W}^{-1}W^\intercal.
\$
Now we will prove Lemma \ref{lem:g} by providing lower and upper bounds of the dual Hessian $\nabla^2 g(\cdot)$. 
 Following from Assumption \ref{ass:hessian}, we have $(\overline{H} + \mu W^\intercal W)^{-1} \preceq (\nabla^2 \tilde{f}(\tilde{x}^*(\lambda)) + \mu W^\intercal W)^{-1} \preceq (\underline{H} + \mu W^\intercal W)^{-1}$, where $\underline{H} = \diag\{0, m,\cdots, m\}\otimes I_d$ and $\overline{H} = \diag\{0, \ell,\cdots, \ell\}\otimes I_d$ $\in \RR^{(n+1)d\times(n+1)d}$ with $m = \min_{i\in[n]}\{m_i\}$ and $\ell = \max_{i\in[n]}\{\ell_i\}$.

For any $s>0$, let $S = \diag\{0, s,\cdots, s\}\otimes I_d\in \RR^{(n+1)d\times (n+1)d}$. By the inverse of a block matrix using Schur complement \citep{boyd2004convex}, we have
\$
\rbr{S + \mu W^\intercal W}^{-1} = \begin{pmatrix}
\mu n I_d & -\mu \mathbbm{1}_n^\intercal\otimes I_d \\
-\mu\mathbbm{1}_n\otimes I_d & (\mu + s) I_{nd}
\end{pmatrix}^{-1} = \begin{pmatrix}
\frac{\mu +s}{\mu ns} I_d & \frac{1}{ns} \mathbbm{1}_n^\intercal\otimes I_d \\
\frac{1}{ns}\mathbbm{1}_n\otimes I_d & \frac{1}{\mu +s} I_{nd} + \frac{\mu}{s(\mu + s)}Z
\end{pmatrix}.
\$
Moreover, by matrix multiplication, we have
\$
W\rbr{S + \mu W^\intercal W}^{-1}W^\intercal = \frac{1}{\mu+s}I_{nd} + \frac{s}{\mu(\mu + s)}Z.
\$

Also note that for any matrix $S\in\RR^{(n+1)d\times(n+1)d}$, if $\underline{S} \preceq S\preceq \overline{S}$, we have $W\underline{S}W^\intercal \preceq WSW^\intercal \preceq W\overline{S}W^\intercal$.
Thus, using the fact that $0\preceq Z \preceq I$, we have 
\$
\frac{1}{\mu + s} I_{nd} \preceq W\rbr{S + \mu W^\intercal W}^{-1}W^\intercal  \preceq \frac{1}{\mu} I_{nd}.
\$
Thus, we obtain the following lower and upper bounds on the dual Hessian,
\$
\frac{1}{\mu + \ell} I_{nd} \preceq W\rbr{\overline{H} + \mu W^\intercal W}^{-1}W^\intercal\preceq -\nabla^2 g(\lambda) \preceq W\rbr{\underline{H} + \mu W^\intercal W}^{-1}W^\intercal \preceq \frac{1}{\mu} I_{nd}.
\$
Therefore, we conclude that the dual function $g(\cdot)$ is $m_g$-strongly concave with $m_g = {1}/{(\mu + \ell)}$ and $\nabla g(\cdot)$ is $\ell_g$-Lipschitz continuous with $\ell_g ={1}/{\mu}$.
\end{proof}
We remark that there are existing literature \citep{nesterov2005smooth,beck2014fast, uribe2020dual} studying dual problems and showing the strong concavity and Lipschitz gradients of the dual function $g(\cdot)$. While here we provide tighter bounds with constants $m_g$ and $\ell_g$ defined in Lemma \ref{lem:g} utilizing the structure of the server-client topology. 

\begin{lemma}\label{lem:L_property1}
Under Assumption \ref{ass:hessian}, for any fixed $\lambda\in\RR^{nd}$, the function $L(x,\lambda)$ is convex with respect to $x$ and its partial gradient $\nabla_x L(x,\lambda)$ is $\ell_L$-Lipschitz continuous with constant $\ell_L = \ell+\mu$. 
\end{lemma}

\begin{proof}
Taking partial Hessian with respect to $x$ of the function $L$ defined in \eqref{eq:a2}, we have
\$
\nabla^2_{xx} L(x,\lambda) = \nabla^2f(x) + \mu M.
\$ 
Then under Assumption \ref{ass:hessian} and using the fact that $0\preceq M\preceq I_{nd}$, we have
\$
mI_{nd}\preceq \nabla^2_{xx} L(x,\lambda) \preceq (\ell+\mu)I_{nd}.
\$
Thus, we conclude that for any fixed $\lambda\in\RR^{nd}$, the function $L(x,\lambda)$ is convex with respect to $x$ and its partial gradient $\nabla_x L(x,\lambda)$ is $\ell_L$-Lipschitz continuous with constants $\ell_L = \ell+\mu$. 
\end{proof}

\subsection{Analysis of Lemmas \ref{lem:gL} and \ref{lem:Lg}} \label{sect:a4}
In this section, we provide proofs of Lemmas \ref{lem:gL} and \ref{lem:Lg}. Before that, we first introduce the following lemma, which will be used frequently in the subsequent proofs. 
\begin{lemma}\label{lem:mx}
Under Assumption \ref{ass:hessian}, it holds for all $k = 0,1,\cdots,K-1$ that,
\$\norm{W\tilde{x}^k - W\tilde{x}^*(\lambda^k)}^2_{BD^k} \le \frac{\beta}{m^2} \norm{\nabla_xL(x^k,\lambda^k)}^2,\$
where constants $\beta = \max\{\max_{i\in J_1}\{b_i\}, \max_{i\in J_2}\{b_i(\ell_i+\mu)\}\}$ and $\mu$ is the penalty term in $L$ defined in \eqref{eq:a2}.
\end{lemma}
\begin{proof}
By \eqref{eq:a1} and Lemma \ref{lem:x0star}, we have
\#\label{eq:wx}
\norm{W\tilde{x}^k - W\tilde{x}^*(\lambda^k)}^2_{BD^k} 
&=\norm{-M\rbr{x^k - x^*(\lambda^k)} - \frac{1}{\mu}Z\rbr{\lambda^k - \lambda^k}}^2_{BD^k} \notag \\
&=\norm{M\rbr{x^k - x^*(\lambda^k)}}^2_{BD^k}\notag \\
&\le \beta\norm{M\rbr{x^k - x^*(\lambda^k)}}^2 \notag \\
&\le \beta\norm{x^k - x^*(\lambda^k)}^2,
\#
where the first inequality follows from the fact that $\norm{BD^k}\le\beta$ and the last inequality follows from the fact that $\norm{M} =1$.

Now, we aim to upper bound the RHS of \eqref{eq:wx} using $\norm{\nabla_xL(x^k,\lambda^k)}^2$. 
Using Lemma \ref{lem:x_opt}, we have
\#\label{eq:bb1}
\nabla_xL(x^k,\lambda^k) &= \nabla_xL(x^k,\lambda^k) - \nabla_xL({x}^*(\lambda^k),\lambda^k) \notag\\
& = \nabla f(x^k) - \nabla f({x}^*(\lambda^k)) + \mu M\rbr{x^k - {x}^*(\lambda^k)} \notag \\
& =\rbr{\nabla^2 f(z) + \mu M}\rbr{x^k - {x}^*(\lambda^k)},
\#
where $z = \gamma x^k + (1-\gamma){x}^*(\lambda^k)$ with some $0\le\gamma\le1$ and the last equality follows from the mean value theorem. Under Assumption \ref{ass:hessian}, using the fact that $0\preceq M\preceq I_{nd}$, we have
\#\label{eq:bb2}
\nabla^2 f(z) + \mu M &\succeq mI_{nd}.
\#
Combining \eqref{eq:bb1} and \eqref{eq:bb2}, we have
\#\label{eq:bb3}
\norm{\nabla_xL(x^k,\lambda^k)}^2 &= \norm{\rbr{\nabla^2 f(z) + \mu M}\rbr{x^k - {x}^*(\lambda^k)}}^2 \notag \\
&\ge m^2 \norm{x^k - x^*(\lambda^k)}^2.
\#
Thus, combing \eqref{eq:wx} and \eqref{eq:bb3}, we have
\$
\norm{W\tilde{x}^k - W\tilde{x}^*(\lambda^k)}^2_{BD^k} \le \frac{\beta}{m^2} \norm{\nabla_xL(x^k,\lambda^k)}^2.
\$
This concludes the proof of the lemma. 
\end{proof}

Now we provide the proof of Lemma \ref{lem:gL}. 

\begin{proof}[Proof of Lemma \ref{lem:gL}]
Using the $\ell_g$-Lipschitz continuity of $\nabla g(\cdot)$ in Lemma \ref{lem:g}, we have
\#\label{eq:step1}
g(\lambda^{k+1}) &\ge g(\lambda^k) + \inner{\nabla g(\lambda^{k})}{\lambda^{k+1} - \lambda^k} - \frac{\ell_g}{2}\norm{\lambda^{k+1} - \lambda^k}^2 \notag \\
& = g(\lambda^k) + \underbrace{\inner{\nabla g(\lambda^{k})}{BD^kW\tilde{x}^k}}_{\textstyle \text{term (i)}} - \frac{\ell_g}{2}\underbrace{\norm{\lambda^{k+1} - \lambda^k}^2}_{\textstyle \text{term (ii)}},
\#
where the inequality follows from the equivalence of Lipschitz continuity of $ g(\cdot)$ \citep{nesterov2018lectures} and the equality follows from the dual update of $\lambda^k$ in \eqref{algo:reformulate}.

In the sequel, we will provide bounds on terms (i) and (ii), respectively.

\textbf{Term (i). }Adding and subtracting a term $\inner{\nabla g(\lambda^{k})}{BD^k\nabla g(\lambda^{k})}$ in term (i), we have 
\#\label{eq:lambdainner}
\inner{\nabla g(\lambda^{k})}{BD^kW\tilde{x}^k} & = \inner{\nabla g(\lambda^{k})}{BD^kW\tilde{x}^k - BD^k\nabla g(\lambda^{k})} + \inner{\nabla g(\lambda^{k})}{BD^k\nabla g(\lambda^{k})} \notag \\
& \ge -\frac{1}{2}\norm{\nabla g(\lambda^{k})}^2_{BD^k} - \frac{1}{2}\norm{W\tilde{x}^k - \nabla g(\lambda^{k})}^2_{BD^k} + \norm{\nabla g(\lambda^{k})}^2_{BD^k} \notag \\
& = \frac{1}{2}\norm{\nabla g(\lambda^{k})}^2_{BD^k} - \frac{1}{2}\norm{W\tilde{x}^k - W\tilde{x}^*(\lambda^k)}^2_{BD^k}, 
\#
where the inequality follows from Cauchy-Schwarz inequality and the last equality follows from Lemma \ref{lem:dual_hessian}.

\textbf{Term (ii). }Consider the dual update in \eqref{algo:reformulate}, we have
\#\label{eq:term2}
\norm{\lambda^{k+1} - \lambda^k}^2 & = \norm{BD^kW\tilde{x}^k}^2 \notag\\
& \le 2\norm{BD^kW\tilde{x}^k - BD^kW\tilde{x}^*(\lambda^k)}^2 + 2 \norm{BD^kW\tilde{x}^*(\lambda^k)}^2 \notag \\
& = 2\norm{BD^kW\tilde{x}^k - BD^kW\tilde{x}^*(\lambda^k)}^2 + 2 \norm{BD^k\nabla g(\lambda^k)}^2 \notag \\
& \le 2\beta\norm{W\tilde{x}^k - W\tilde{x}^*(\lambda^k)}^2_{BD^k} + 2 \beta\norm{\nabla g(\lambda^k)}^2_{BD^k},
\#
where $\beta = \max\{\max_{i\in J_1}\{b_i\}, \max_{i\in J_2}\{b_i(\ell_i+\mu)\}\}$, the first inequality follows from Cauchy-Schwarz inequality, and the last equality follows from Lemma \ref{lem:dual_hessian} and the fact that $\norm{BD^k} = \max_{i\in[n]}\{b_i\norm{D_i^k}\}\le \beta$.

Substituting \eqref{eq:lambdainner} and \eqref{eq:term2} into \eqref{eq:step1}, we have 
\#\label{eq:preresult1}
g(\lambda^{k+1}) &\ge g(\lambda^k) +\rbr{\frac{1}{2}-\beta\ell_g}\norm{\nabla g(\lambda^{k})}^2_{BD^k} - \rbr{\frac{1}{2}+\beta\ell_g}\norm{W\tilde{x}^k - W\tilde{x}^*(\lambda^k)}^2_{BD^k}.
\#
Next, we use Lemma \ref{lem:mx} to upper bound the last term in \eqref{eq:preresult1} with an alternative primal tracking error $\norm{\nabla_xL(x^k,\lambda^k)}^2$. Following from Lemma \ref{lem:mx}, subtracting the dual optimal value $g(\lambda^*)$ and taking a negative sign on both sides of \eqref{eq:preresult1}, we have
\$
\Delta_{\lambda}^{k+1} \le \Delta_{\lambda}^{k} - \rbr{\frac{1}{2}-\beta\ell_g}\norm{\nabla g(\lambda^{k})}^2_{BD^k} + \rbr{\frac{1}{2}+\beta\ell_g}\frac{ \beta}{m^2} \norm{\nabla_xL({x}^k,\lambda^k)}^2,
\${}
where $\Delta_{\lambda}^k$ is defined in \eqref{eq:terrors}.
This concludes the proof of the lemma.
\end{proof}

Next, we provide the proof of Lemma \ref{lem:Lg}. 

\begin{proof}[Proof of Lemma \ref{lem:Lg}]
Consider the tracking error of the primal updates $\Delta_{\tilde x}^k$ defined in \eqref{eq:terrors}, we have
\#\label{eq:deltaL}
\Delta_{x}^{k+1} &=  L(x^{k+1}, \lambda^{k+1}) - L(x^*(\lambda^{k+1}), \lambda^{k+1}) \notag \\
&= \underbrace{L(x^{k+1}, \lambda^{k+1}) -L(x^{k+1}, \lambda^{k})}_{\textstyle \text{term (A)}}+ \underbrace{L(x^{k+1}, \lambda^{k}) -L(x^*(\lambda^{k}), \lambda^{k})}_{\textstyle \text{term (B)}}  \\
& \qquad + \underbrace{L(x^*(\lambda^{k}), \lambda^{k}) - L(x^*(\lambda^{k+1}), \lambda^{k+1})}_{\textstyle \text{term (C)}}. \notag
\#
We remark that here term (A) measures the increase due to dual update, term (B) represents updated primal tracking error, and term (C) shows the difference between dual optimality gaps.
In the sequel, we will provide upper bounds on terms (A)-(C), respectively.

\textbf{Term (A).} Consider the function $L$ defined in \eqref{eq:a2}, we have
\#\label{eq:termA}
&L(x^{k+1}, \lambda^{k+1}) -L(x^{k+1}, \lambda^{k}) \notag\\
&\quad = \rbr{\lambda^{k+1}-\lambda^{k}}^\intercal\rbr{-Mx^{k+1}} - \frac{1}{2\mu}\rbr{\lambda^{k+1}}^\intercal Z\lambda^{k+1} + \frac{1}{2\mu}\rbr{\lambda^{k}}^\intercal Z\lambda^{k} \notag\\
&\quad = \rbr{\lambda^{k+1}-\lambda^{k}}^\intercal\rbr{-Mx^{k+1}-\frac{1}{\mu}Z\lambda^{k+1}} +\frac{1}{\mu}\rbr{\lambda^{k+1}-\lambda^{k}}^\intercal Z\lambda^{k+1} \notag\\
& \quad \qquad - \frac{1}{2\mu}\rbr{\lambda^{k+1}}^\intercal Z\lambda^{k+1} + \frac{1}{2\mu}\rbr{\lambda^{k}}^\intercal Z\lambda^{k} \notag\\
&\quad = \underbrace{\rbr{\lambda^{k+1}-\lambda^{k}}^\intercal W\tilde{x}^{k+1}}_{\textstyle \text{term (A.1)}} +\frac{1}{2\mu}\underbrace{\rbr{\lambda^{k+1} - \lambda^k}^{\intercal}Z\rbr{\lambda^{k+1} - \lambda^k}}_{\textstyle \text{term (A.2)}},
\#
where the last equality follows from \eqref{eq:a1}. Next, we will bound terms (A.1) and (A.2), respectively.

\textbf{Term (A.1).} Based on the dual updates in Algorithm \ref{algo:fedh}, we have
\#\label{eq:terma1}
\rbr{\lambda^{k+1}-\lambda^{k}}^\intercal W\tilde{x}^{k+1} &= \rbr{BD^kW\tilde{x}^{k}}^\intercal W\tilde{x}^{k+1} \notag\\
& = \norm{W\tilde{x}^{k}}^2_{BD^k} + \rbr{W\tilde{x}^{k}}^\intercal BD^k\rbr{W\tilde{x}^{k+1} - W\tilde{x}^{k}} \notag\\
& \le \norm{W\tilde{x}^{k}}^2_{BD^k} + \frac{1}{2}\norm{W\tilde{x}^{k}}_{BD^k}^2 + \frac{1}{2}\norm{W\tilde{x}^{k+1} - W\tilde{x}^{k}}_{BD^k}^2 \notag\\
& = \frac{3}{2}\underbrace{\norm{W\tilde{x}^{k}}_{BD^k}^2}_{\textstyle \text{term (A.1.1)}} + \frac{1}{2}\underbrace{\norm{W\tilde{x}^{k+1} - W\tilde{x}^{k}}_{BD^k}^2}_{\textstyle \text{term (A.1.2)}} , 
\#
where the inequality follows from Cauchy-Schwarz inequality. Now we bound upper terms (A.1.1) and (A.1.2), respectively.

\textbf{Term (A.1.1).} Using Cauchy-Schwarz inequality, we have
\#\label{eq:whatx}
\norm{W\tilde{x}^{k}}^2_{BD^k} & \le 2\norm{W\tilde{x}^{k} - W\tilde{x}^*(\lambda^k)}^2_{BD^k} + 2\norm{W\tilde{x}^*(\lambda^k)}^2_{BD^k}  \notag\\
&= 2\norm{W\tilde{x}^{k} - W\tilde{x}^*(\lambda^k)}^2_{BD^k} + 2\norm{\nabla g(\lambda^k)}^2_{BD^k}.
\#
By using Lemma \ref{lem:mx} and following from \eqref{eq:whatx}, we have
\#\label{eq:wx_bd}
\norm{W\tilde{x}^{k}}^2_{BD^k} \le \frac{ 2\beta}{m^2} \norm{\nabla_xL({x}^k,\lambda^k)}^2 + 2\norm{\nabla g(\lambda^k)}^2_{BD^k}.
\#
\textbf{Term (A.1.2).}  Based on \eqref{eq:a1}, we have
\#\label{eq:a1ii}
\norm{W\tilde{x}^{k+1} - W\tilde{x}^{k}}_{BD^k}^2 &= \norm{-M\rbr{x^{k+1} - x^k} - \frac{1}{\mu}Z\rbr{\lambda^{k+1} - \lambda^k}}_{BD^k}^2 \notag \\
&\le 2 \norm{M\rbr{x^{k+1} - x^k}}_{BD^k}^2 +2\norm{\frac{1}{\mu}Z\rbr{\lambda^{k+1} - \lambda^k}}_{BD^k}^2 \notag \\
&\le 2 \norm{x^{k+1} - x^k}_{BD^k}^2 +\frac{2}{\mu^2}\norm{\lambda^{k+1} - \lambda^k}_{BD^k}^2,
\#
where the first inequality follows from Cauchy-Schwarz inequality and the last inequality follows from the fact that $\norm{M} = \norm{Z} = 1$.

Based on the fact that $\norm{BD^k}\le\beta$ and the primal update in \eqref{algo:reformulate}, we have
\#\label{eq:x_difference}
\norm{x^{k+1} - x^k}_{BD^k}^2 \le \beta\norm{x^{k+1} - x^k}^2 =  \beta\norm{A(D^k)^{-1}\nabla_xL(\tilde{x}^k,\lambda^k)}^2.
\#
By combining \eqref{eq:term2} and Lemma \ref{lem:mx}, we have
\#\label{eq:lam1}
\norm{\lambda^{k+1} - \lambda^k}^2_{BD^k}&\le \beta\norm{\lambda^{k+1} - \lambda^k}^2 \notag\\
&\le \frac{2 \beta^3}{m^2}\norm{\nabla_xL({x}^k,\lambda^k)}^2 + 2 \beta^2\norm{\nabla g(\lambda^k)}^2_{BD^k}.
\#
Thus, substituting \eqref{eq:x_difference} and \eqref{eq:lam1} into \eqref{eq:a1ii}, we have
\#\label{eq:a1ii_final}
& \norm{W\tilde{x}^{k+1} - W\tilde{x}^{k}}_{BD^k}^2  \\
&\quad \le 2  \beta\norm{A(D^k)^{-1}\nabla_xL(\tilde{x}^k,\lambda^k)}^2 
 + \frac{4 \beta^3}{m^2\mu^2}\norm{\nabla_xL({x}^k,\lambda^k)}^2 + \frac{4 \beta^2}{\mu^2}\norm{\nabla g(\lambda^k)}^2_{BD^k}.\notag
\#
Thus, substituting \eqref{eq:wx_bd} and \eqref{eq:a1ii_final} into \eqref{eq:terma1}, we have
\#\label{eq:term_a_1_result}
\rbr{\lambda^{k+1}-\lambda^{k}}^\intercal W\tilde{x}^{k+1} \le & \frac{\beta\rbr{3\mu^2 + 2 \beta^2}}{m^2\mu^2}\norm{\nabla_xL({x}^k,\lambda^k)}^2 \\
& + \rbr{3 + \frac{2\beta^2}{\mu^2}}\norm{\nabla g(\lambda^k)}^2_{BD^k} + \beta\norm{A(D^k)^{-1}\nabla_xL(x^k,\lambda^k)}^2.\notag
\#
\textbf{Term (A.2).} Using the fact that $\norm{Z}=1$ and combining \eqref{eq:term2} and Lemma \ref{lem:mx}, we have
\#\label{eq:term_a2}
\rbr{\lambda^{k+1} - \lambda^k}^{\intercal}Z\rbr{\lambda^{k+1} - \lambda^k} &\le \norm{Z}\norm{\lambda^{k+1} - \lambda^k} ^2\notag\\
&\le  \frac{ 2\beta^2}{m^2} \norm{\nabla_xL({x}^k,\lambda^k)}^2 + 2\beta\norm{\nabla g(\lambda^k)}^2_{BD^k}.
\#
Therefore, substituting \eqref{eq:term_a_1_result} and \eqref{eq:term_a2} into \eqref{eq:termA}, we have
\#\label{eq:terma}
&L(x^{k+1}, \lambda^{k+1}) -L(x^{k+1}, \lambda^{k}) \notag \\ 
&\quad\le \frac{\beta\rbr{3\mu^2 + 2 \beta^2+\beta\mu}}{m^2\mu^2}\norm{\nabla_xL({x}^k,\lambda^k)}^2 + \rbr{3 + \frac{2\beta^2}{\mu^2} + \frac{\beta}{\mu}}\norm{\nabla g(\lambda^k)}^2_{BD^k} \\
& \qquad \quad +\beta\norm{A(D^k)^{-1}\nabla_xL(\tilde{x}^k,\lambda^k)}^2. \notag
\#
This provides an upper bound on term (A).

\textbf{Term (B).} Using the $\ell_L$-Lipschitz continuity of $\nabla_xL(x,\lambda^k)$ from Lemma \ref{lem:L}, we have
\#\label{eq:b1}
L(x^{k+1}, \lambda^{k}) &\le L(x^{k}, \lambda^{k}) + \nabla_x L(x^{k}, \lambda^{k})^\intercal\rbr{x^{k+1} - x^k} + \frac{\ell_L}{2}\norm{x^{k+1} - x^k}^2\\
&= L(x^{k}, \lambda^{k}) - \nabla_x L(x^{k}, \lambda^{k})^\intercal A\rbr{D^k}^{-1}\nabla_xL(x^k,\lambda^k) + \frac{\ell_L}{2}\norm{A(D^k)^{-1}\nabla_xL(x^k,\lambda^k)}^2, \notag
\#
where the equality follows from the primal updates in Algorithm \ref{algo:reformulate}. Subtracting $L(x^{*}(\lambda^k), \lambda^k)$ on both sides of \eqref{eq:b1}, we have the following upper bound on term (B):
\#\label{eq:termb}
L(x^{k+1}, \lambda^{k}) - L(x^{*}(\lambda^k), \lambda^k)\le \Delta_{x}^k -
\norm{ \nabla_x L(x^{k}, \lambda^{k})}^2_{A\rbr{D^k}^{-1}} + \frac{\ell_L}{2}\norm{A(D^k)^{-1}\nabla_xL(x^k,\lambda^k)}^2.
\# 
\textbf{Term (C).} Using the dual function $g(\lambda^k)$ defined in Problem \ref{prob:dual} and the dual optimality gap $\Delta_{\lambda}^k$ defined in \eqref{eq:terrors}, we have
\#\label{eq:termc}
L(x^*(\lambda^{k}), \lambda^{k}) - L(x^*(\lambda^{k+1}), \lambda^{k+1}) = g(\lambda^k) - g(\lambda^{k+1}) = \Delta_{\lambda}^k - \Delta_{\lambda}^{k+1}. 
\#
Substituting \eqref{eq:terma}, \eqref{eq:termb}, and \eqref{eq:termc} into \eqref{eq:deltaL}, we have
\#\label{step2}
\Delta_{x}^{k+1} \le& \Delta_{x}^k  +\rbr{3 + \frac{2\beta^2}{\mu^2}+\frac{\beta}{\mu}}\norm{\nabla g(\lambda^k)}^2_{BD^k}  +\Delta_{\lambda}^k - \Delta_{\lambda}^{k+1}\\
& - \nabla_x L(x^{k}, \lambda^{k})^\intercal\rbr{A(D^k)^{-1} - \rbr{\beta+\frac{\ell_L}{2}}A^2(D^k)^{-2} -  \frac{\beta\rbr{3\mu^2 + 2 \beta^2+\beta\mu}}{m^2\mu^2}I}\nabla_x L(x^{k}, \lambda^{k}). \notag 
\#
Finally, using \eqref{eq:dx_eq}, we conclude the proof of the lemma.
\end{proof}

\section{Proof of Theorems}\label{sect:b}

In this section, we provide the proof of Theorem \ref{thm:strong}. Before that, we first 
introduce the following lemma with some pivotal results derived from condition of stepsizes \eqref{eq:steps}. For convenience, for any symmetric matrix $X$, we denote by $\theta_{\min}(X)$ its smallest eigenvalue.

\begin{lemma} \label{lem:b11} Under Assumption \ref{ass:hessian}, we suppose that the stepsizes satisfy \eqref{eq:steps}. It holds that
\$
\beta \le \min\left\{\frac{\mu}{9}, \frac{\underline{\alpha}m^2}{21} \right\}, \quad \norm{A(D^k)^{-1}} \le \frac{1}{2\rbr{2\beta + \ell_L}}.
\$
Moreover, consider constants $\underline{\alpha}$ and $\underline{\beta}$ defined in \ref{thm:strong}, we have,
\$
\underline{\alpha} \le \theta_{\min}(A(D^k)^{-1}), \quad \underline{\beta} \le \theta_{\min}(BD^k).
\$
\end{lemma}
\begin{proof}
First note that when dual stepsizes satisfy \eqref{eq:steps}, we have
\$
\beta = \max\left\{\max_{i\in J_1}{b_i}, \max_{i\in J_2}b_i(\ell_i+\mu)\right\} \le \min\left\{\frac{\mu}{9}, \frac{\underline{\alpha}m^2}{21} \right\}.
\$
If primal stepsizes satisfy \eqref{eq:steps}, using $\beta\le \mu/9$ and $\ell_L = \mu + \ell$ defined in Lemma \ref{lem:L_property1}, we have
\$
\norm{A(D^k)^{-1}} = \max_{i\in[n]} a_i\norm{(D_i^k)^{-1}} \le \max\left\{\max_{i\in J_1}{a_i}, \max_{i\in J_2}\frac{a_i}{m_i+\mu}\right\} \le \frac{1}{2\rbr{\frac{11}{9}\mu + \ell}} \le \frac{1}{2(2\beta + \ell_L)}.
\$
Moreover, by the definition of $\underline{\alpha}$ and $\underline{\beta}$ in Theorem \ref{thm:strong}, we have
\$
&\underline{\alpha}  = \min\{\min_{i\in J_1}\{a_i\}, \min_{i\in J_2}\{a_i/(\ell_i+\mu)\}\}\le\min_{i\in[n]}a_i\theta_{\min}((D_i^k)^{-1}) = \theta_{\min}(A(D^k)^{-1}),\\
&\underline{\beta} = \min\{\min_{i\in J_1}\{b_i\}, \min_{i\in J_2}\{b_i(m_i+\mu)\}\}\le\min_{i\in[n]}b_i\theta_{\min}(D_i^k) = \theta_{\min}(BD^k).
\$
\end{proof}

Next, we provide some bounds on a constant and a matrix related to $\kappa$ and  $P^k$ defined in Lemma \ref{lem:Lg}, respectively. They will be used in the subsequent analysis. 
\begin{lemma}\label{lem:b12}
Under Assumption \ref{ass:hessian}, we suppose that the stepsizes satisfy \eqref{eq:steps}. It holds that
\$
\kappa + 12\ell_g\beta \le \frac{9}{2}, \quad \theta_{\min}(Q^k) \ge \frac{\underline{\alpha}}{4},
\$
where the matrix $Q^k = P^k - (6 + 12\beta\ell_g)\beta/m^2 I$, the constant $\kappa$ and the matrix $P^k$ are defined in Lemma \ref{lem:Lg}, the Lipschitz constant $\ell_g$ is defined in Lemma \ref{lem:g}, and $\underline{\alpha}$ is defined in Theorem \ref{thm:strong}.
\end{lemma}
\begin{proof}
By direct calculation, when $\beta\le\mu/9$ as assumed in \eqref{eq:steps}, with $\ell_g = 1/\mu$, we have 
\#\label{eq:constant}
\kappa + 12\ell_g\beta = 3 + 2\beta^2/\mu^2 + 13\beta/\mu \le 9/2.
\#
Then consider the matrix $Q^k = P^k - (6 + 12\beta\ell_g)\beta/m^2 I$, we have
\# \label{eq:qk1}
Q^k & = A(D^k)^{-1} - \rbr{\beta+\frac{\ell_L}{2}}A^2(D^k)^{-2}- \frac{(\kappa + 12\ell_g\beta + 6)\beta}{m^2}I \notag \\
& \succeq A(D^k)^{-1} - \rbr{\beta+\frac{\ell_L}{2}}A^2(D^k)^{-2}- \frac{21\beta}{2m^2}I \notag \\
& \succeq \frac{1}{2}A(D^k)^{-1} - \rbr{\beta+\frac{\ell_L}{2}}A^2(D^k)^{-2}\notag \\
& = \frac{1}{2} A^{1/2}(D^k)^{-1/2}\sbr{I- \rbr{2\beta+\ell_L}A(D^k)^{-1}}A^{1/2}(D^k)^{-1/2},
\#
where the first inequality follows from the condition that $\beta\le\mu/9$ and \eqref{eq:constant} and the last inequality follows from the condition that $\beta\le\underline{\alpha}m^2/21$ and Lemma \ref{lem:b11}. 

Thus, consider the smallest eigenvalue to the matrix $Q^k$, following from \eqref{eq:qk1}, we have 
\$
\theta_{\min}(Q^k) &\ge \frac{1}{2} \theta_{\min}\rbr{A^{1/2}(D^k)^{-1/2}\sbr{I- \rbr{2\beta+\ell_L}A(D^k)^{-1}}A^{1/2}(D^k)^{-1/2}} \notag\\
&\ge \frac{1}{2} \sbr{1- \rbr{2\beta+\ell_L}\norm{A(D^k)^{-1}}} \theta_{\min}\rbr{A(D^k)^{-1}} \notag \\
&\ge \frac{1}{4} \theta_{\min}\rbr{A(D^k)^{-1}} \notag \\
& \ge \frac{\underline{\alpha}}{4},
\$
where the third inequality follows from Lemma \ref{lem:b11} and the condition that $\beta\le\mu/9$ and the last inequality follows from Lemma \ref{lem:b11}.
\end{proof}

Now we provide the proof of Theorem \ref{thm:strong}.
\subsection{Proof of Theorem \ref{thm:strong}}\label{proof:thm:strong} 

\begin{proof}
By multiplying Lemma \ref{lem:gL} by $12$ and adding Lemma \ref{lem:Lg}, we have
\#\label{eq:step_1_and_2}
&13\Delta_{\lambda}^{k+1} + \Delta_{x}^{k+1}\\ 
&\quad \le  13\Delta_{\lambda}^{k}  - \rbr{6-12\ell_g\beta - \kappa}\norm{\nabla g(\lambda^{k})}^2_{BD^k}+ \Delta_{x}^k - \norm{\nabla_x L(x^{k}, \lambda^{k})}^2_{Q^k}, \notag
\#
where the matrix $Q^k = P^k - (6 + 12\beta\ell_g)\beta/m^2 I$.

Following from the $m_g$-strong concavity of $g(\lambda)$ in Lemma \ref{lem:g} with $m_g = 1/(\mu+{\ell})$, we have
\#\label{eq:g1}
\norm{\nabla g(\lambda^k)}^2 \ge 2m_g \rbr{g(\lambda^*) - g(\lambda^k)} = \frac{2}{\mu + \ell} \Delta_{\lambda}^k.
\# 
Thus, when taking $\beta\le \mu/9$, following from Lemma \ref{lem:b12}, we have
\# \label{eq:gap_k} 
\rbr{6-12\ell_g\beta - \kappa}\norm{\nabla g(\lambda^{k})}^2_{BD^k} 
&\ge \frac{3}{2}\norm{\nabla g(\lambda^{k})}^2_{BD^k} \notag \\
&\ge \frac{3}{2}\theta_{\min}(BD^k)\norm{\nabla g(\lambda^{k})}^2 \notag \\
& \ge \frac{3\underline{\beta}}{(\mu+{\ell})}\Delta_{\lambda}^k,
\#
where the last inequality follows from Lemma \ref{lem:b11} and \eqref{eq:g1}.

Similarly, following from the $m$-strong convexity of $L(\cdot, \lambda^k)$ in Lemma \ref{lem:L_property1}, we have
\#\label{eq:s2}
\norm{\nabla_x L({x}^{k}, \lambda^{k})}^2 \ge 2m\rbr{L(x^k, \lambda^k) - L(x^*(\lambda^k), \lambda^k)} = 2m\Delta_{x}^k.
\#
Thus, when taking $\beta\le\mu/9$, we have
\#\label{eq:l_k}
\norm{\nabla_x L(x^{k}, \lambda^{k})}^2_{Q^k} &  \ge \theta_{\min}(Q^k)\norm{\nabla_x L(x^{k}, \lambda^{k})}^2 \notag \\
&  \ge \frac{ \underline{\alpha}}{4}\norm{\nabla_x L(x^{k}, \lambda^{k})}^2 \notag \\
&  \ge  \frac{ \underline{\alpha}}{2}m\Delta_{x}^k,
\#
where the second inequality follows from Lemma \ref{lem:b12} and the last inequality follows from \eqref{eq:s2}.

Finally, substituting \eqref{eq:gap_k} and \eqref{eq:l_k} into \eqref{eq:step_1_and_2},  we have
\$
13\Delta_{\lambda}^{k+1} + \Delta_{x}^{k+1} 
& \le  13\Delta_{\lambda}^{k}  - \frac{3\underline{\beta}}{\mu+{\ell}}\Delta_{\lambda}^k + \rbr{1 -\frac{m\underline{\alpha}}{2}}\Delta_{x}^k \\
&\le 13\rbr{1-\frac{3\underline{\beta}}{13(\mu+{\ell})}}\Delta_{\lambda}^{k} + \rbr{1 - \frac{m\underline{\alpha}}{2}}\Delta_{x}^k \notag \\
&\le \rbr{1-\rho}\rbr{13\Delta_{\lambda}^{k} + \Delta_{x}^k},
\$
where $\rho = \min\{\frac{3\underline{\beta}}{13(\mu+\ell)}, \frac{m\underline{\alpha}}{2} \}$. Thus, let $\Delta^k = 13\Delta_{\lambda}^{k} + \Delta_{\tilde x}^k$, following from \eqref{eq:dx_eq}, we have
\$
\Delta^{k+1} \le (1-\rho)\Delta^k.
\$
This concludes the proof of our theorem.
\end{proof}

\section{Details for Numerical Experiments} \label{sect:experiment_details}

In this section, we present more details on experimental results for FedHybrid method. In particular, we consider least squares problems and binary classification problems in a server-client network, both over synthetic and real-life datasets. All the experiments are conducted on 3.30GHz Intel Core i9 CPUs, Ubuntu 20.04.2, in Python 3.8.5. 

\subsection{Experimental Setup} \label{sect:c1}
We evaluate all algorithms on four setups with non-IID data partitioning: (1) Linear regression on a synthetic dataset. (2) Linear regression on a non-IID partitioned Boston housing prices dataset. (3) Regularized logistic regression on a synthetic dataset. (4) Regularized logistic regression on a non-IID partitioned mushroom dataset. We introduce the four setups as follows.

As an initial study, we consider the regularized linear regression problem in the following form,
\#\label{exp:m1}
\min_{\omega\in\RR^d} \frac{1}{2}\sum_{i=1}^n \frac{1}{N_i} \norm{A_{i} \omega - b_{i}}^2 + \frac{\rho}{2}\norm{\omega}^2,
\#
where $n$ is the number of clients in the network, $d$ is the dimension of the decision variable, $N_i$ is the size of local dataset, $A_{k_i}\in \RR^{N_i\times d}$ is the design matrix at client $i$, $b_{k_i}\in \RR^{N_i}$ is the response vector at client $i$, and $\rho\ge0$ is the penalty parameter. For the regularized linear regression model defined in \eqref{exp:m1}, we consider the following two setups with non-IID data partitioning.
 
(1) Linear regression on a synthetic dataset. We set $n = 20$, $d=30$ and local dataset sizes $N_i\sim\text{lognormal}(4,2) + 50$. To obtain a non-IID data distribution among clients, at each client $i$, we generate a scaling value $\eta_i \sim \cN (0, \gamma)$ with variance $\gamma = 1$ and set $A_i = \eta_i\hat{A}_i$, where $\hat{A}_i\in\RR^{|\cD_i|\times d}$ is a matrix with each element following a uniform distribution over $(0,1]$. We also generate an unknown parameter vector $x_0\in\RR^{d}$, $x_0\sim\cN(0, I_{d})$. Given the design matrix $A_i$ and a parameter $x_0$, the response vector $b_i\in\RR^{N_i}$ is generated as 
$
b_i = A_i x_0 + v_i,
$
where the noise vector $v_i\in\RR^{N_i}$ is generated as $v_i\sim\cN(0,\sigma^2 I_{N_i})$ for $\sigma =0.5$.

(2) Linear regression on a non-IID partitioned Boston housing prices dataset. We use the Boston housing prices dataset from UCI \cite{Dua:2019}. For data preprocessing, we add an additional vector of all ones in the design matrix and normalize it. Thus, we have $d = 14$. We set $n=8$ and local dataset sizes $N_i$ following a uniform distribution. To obtain a non-IID data partitioned among clients, we sort the response in an ascending order and then distribute data points to each client by turns according to $N_i$.

Next, we explore the behavior of algorithms for solving binary classification problems. In this setting, we use the following logistic regression model, 
\#\label{exp:logis}
\min_{\omega\in\RR^d} \sum_{i=1}^n \frac{1}{N_i} \sum_{j=1}^{N_i}\sbr{ - y_i \log h_i - (1-y_i)\log(1-h_i)} + \frac{\rho}{2}\norm{\omega}^2,
\#
where $h_i= 1/(1+\exp(-\omega^\intercal \mathbf{x}_i))$, $n$ is the number of clients in the network, $d$ is the dimension of the decision variable, $N_i$ is the size of local dataset, $\mathbf{x}_i\in\RR^{d\times N_i}$ is the feature matrix at client $i$, $y_{i}\in \RR^{N_i}$ is the label of local data at client $i$, and $\rho$ is the penalty parameter.  For the regularized logistic regression model in \eqref{exp:logis}, we consider the following two setups with non-IID data partitioning.

(3) Regularized logistic regression on a synthetic dataset. The dataset Synthetic($0.5$, $0.5$) is originally introduced in \cite{li2018federated}. We set $n = 20$, $d=30$ and local dataset sizes $N_i\sim\text{lognormal}(4,2) + 50$. To obtain a non-IID data distribution among clients, we set eleven and eight clients to only have data points with label $0$ and $1$, respectively, and one client to have data points with both labels. 

(4) Regularized logistic regression on non-IID partitioned a mushroom dataset. We use the Mushroom dataset from UCI \cite{Dua:2019}. For data preprocessing, we encode categorical features, add an additional vector of all ones and normalize the design matrix. In this way, we have $d = 99$. We set $n=20$ and local dataset sizes $N_i$ following a uniform distribution. To obtain a non-IID data partitioned among clients, we set four and three clients to only have data points with label $0$ and $1$, respectively, and one client to have data points with both labels.

\subsection{Supplementary Experimental Results}\label{sect:c2}

\begin{figure}[h]
     \centering
     \begin{subfigure}[b]{0.495\textwidth}
         \centering
         \includegraphics[width=\textwidth]{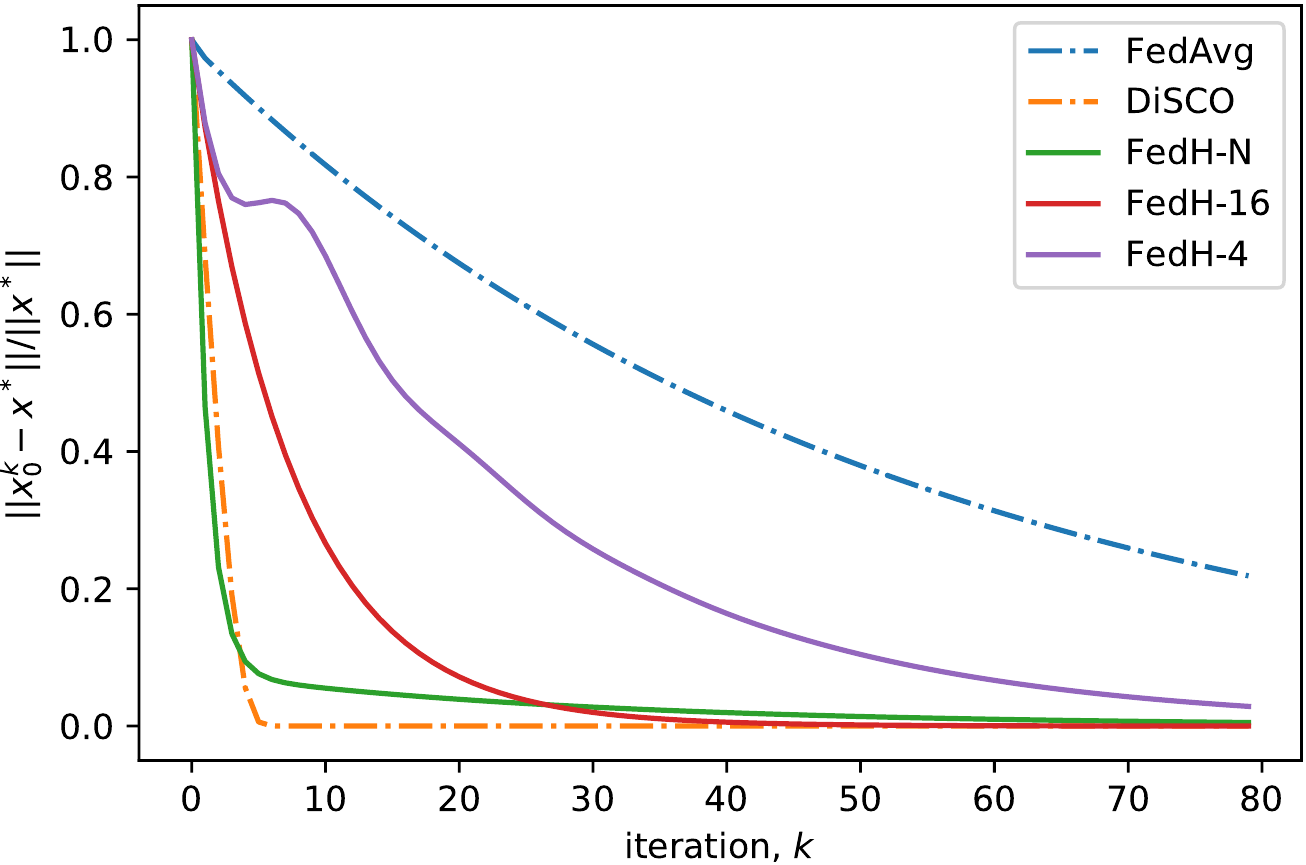}
         \caption{Least squares on synthetic dataset (setup (1))}
     \end{subfigure}
     \begin{subfigure}[b]{0.495\textwidth}
         \centering
         \includegraphics[width=\textwidth]{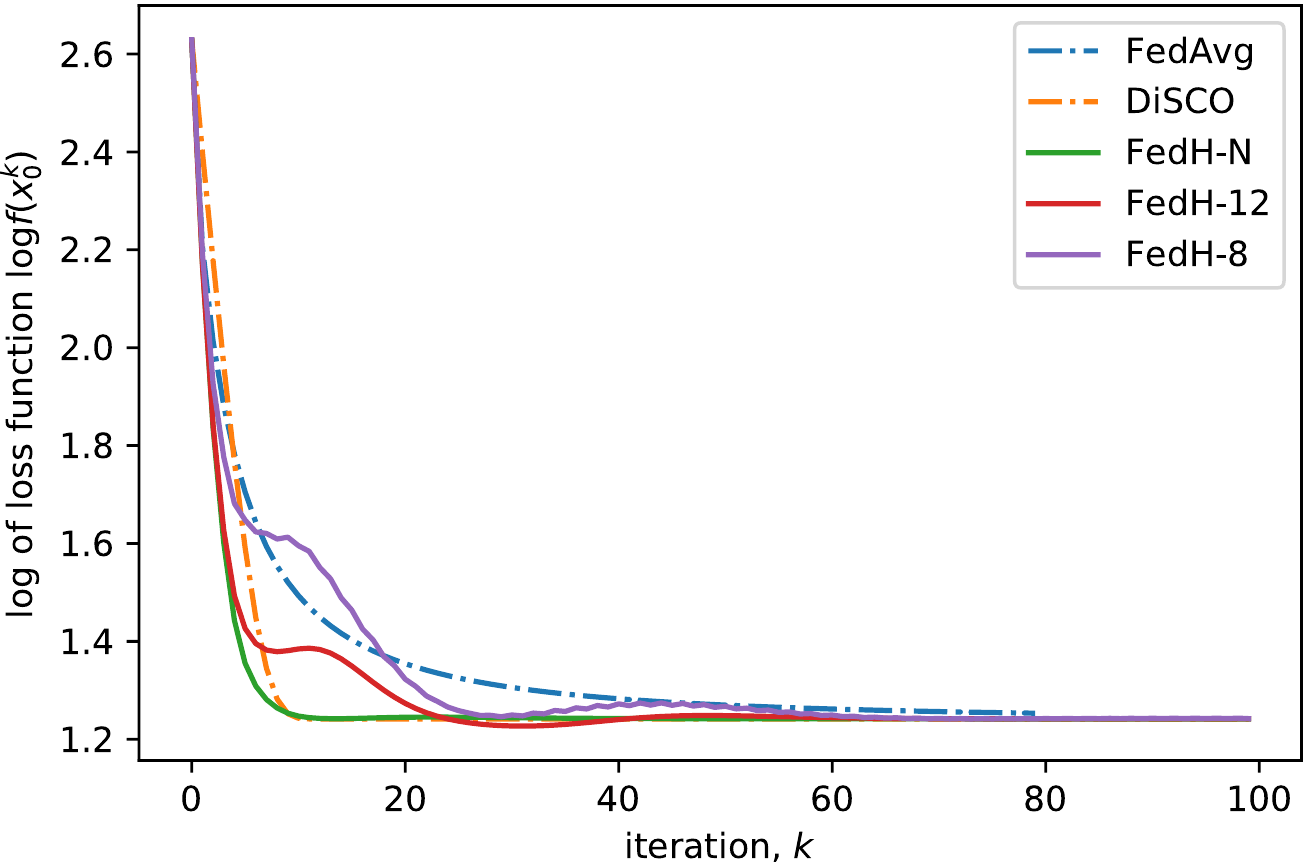}
         \caption{Logistic regression on synthetic dataset (setup (3))}
     \end{subfigure}
     \caption{Comparison of convergence of different algorithms}
     \label{fig:compare_1_app}
\end{figure}

In Figure \ref{fig:compare_1_app}, we provide more results on the comparison of methods FedAvg, DiSCO, and FedHybrid. We remark that with some clients in the network performing Newton-type updates, FedHybrid improves the overall training speed a lot and outperforms the baseline method FedAvg consistently. 

In particular, if all clients in the network perform Newton-type updates, our second-order FedH-N method achieves a comparable convergence performance compared with DiSCO.

In Figure \ref{fig:compare_2_app}, we provide more results on the comparison of convergence rate with different number of Newton-type clients. We remark that as the number of clients that run Newton-type updates increases, the convergence rate of FedHybrid method becomes faster. 
\begin{figure}[h]
     \centering
     \begin{subfigure}[b]{0.495\textwidth}
         \centering
         \includegraphics[width=\textwidth]{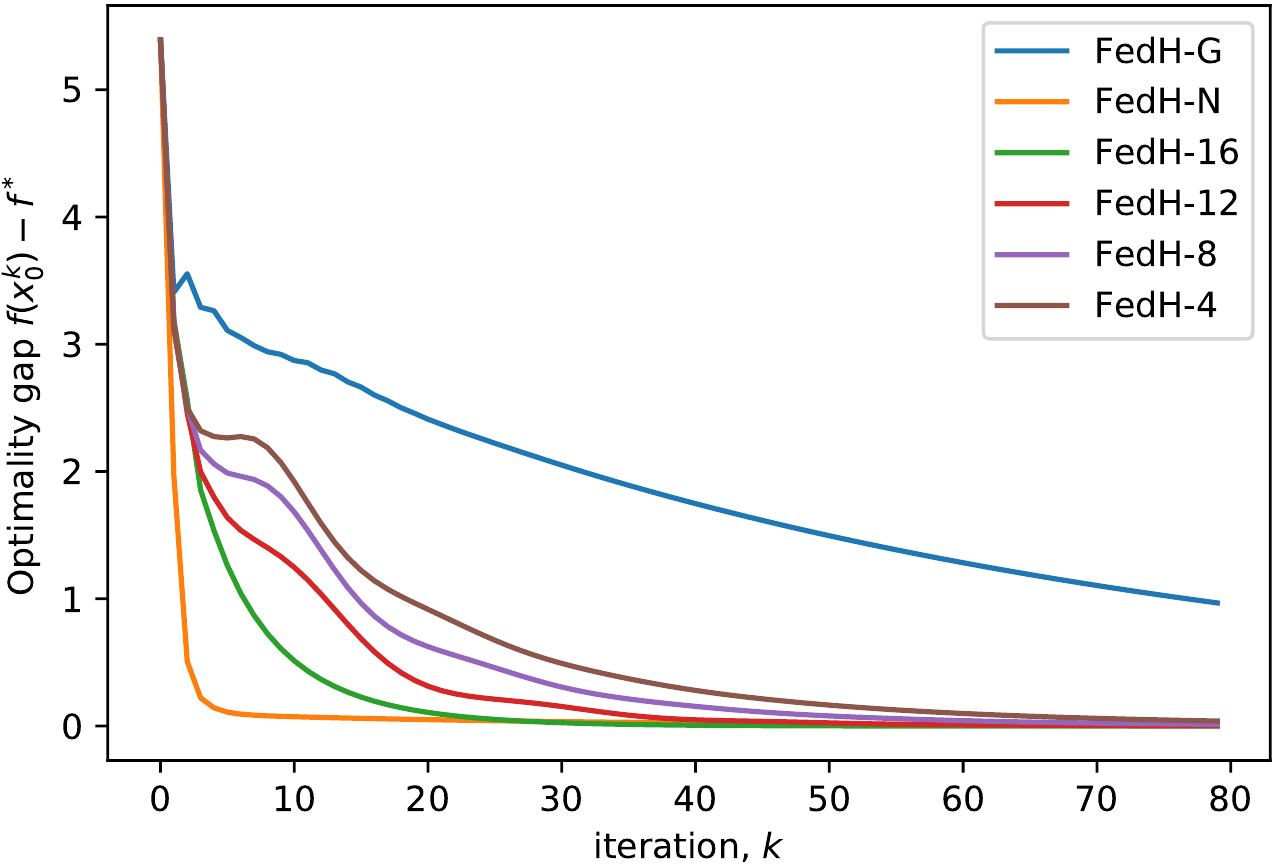}
         \caption{Least squares on synthetic dataset (setup (1)).}
     \end{subfigure}
     \begin{subfigure}[b]{0.495\textwidth}
         \centering
         \includegraphics[width=\textwidth]{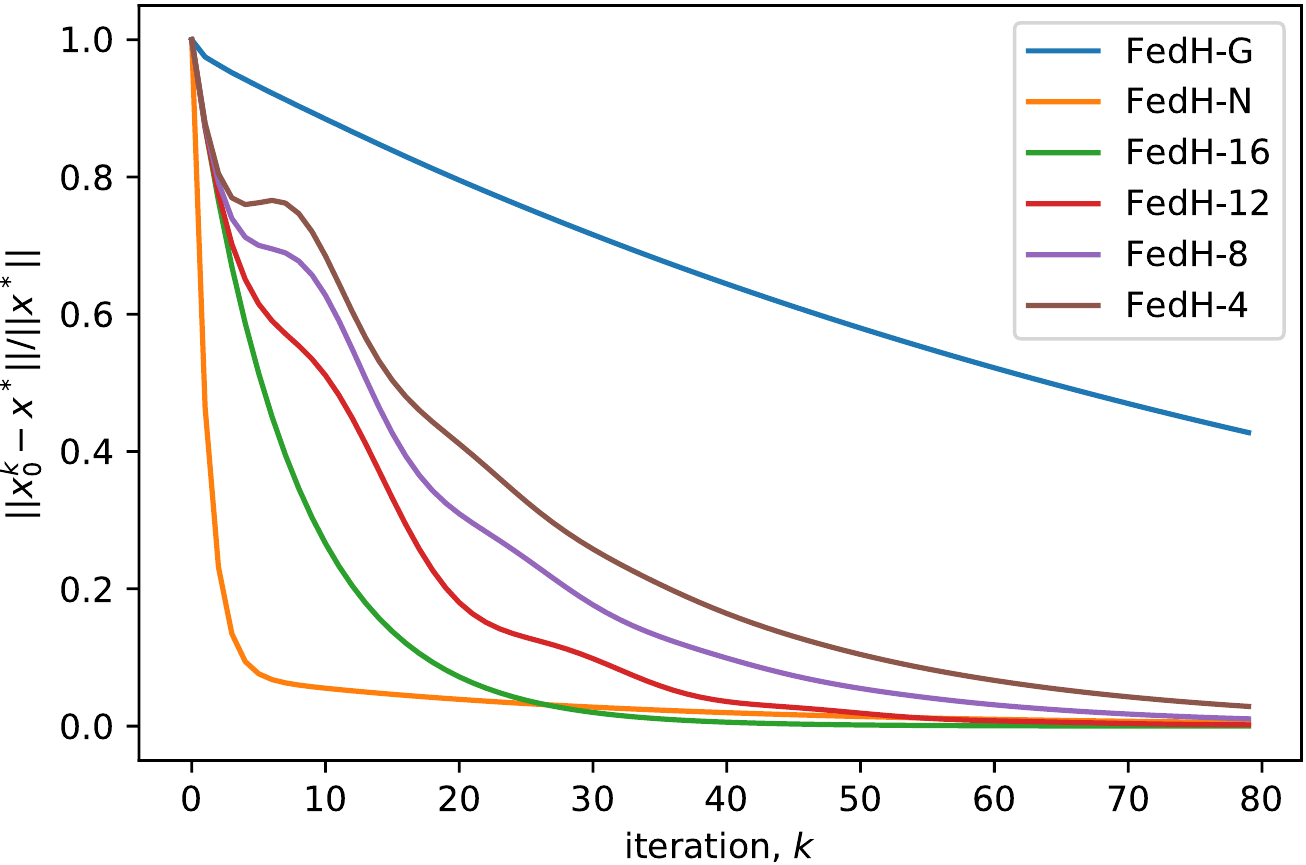}
         \caption{Least squares on synthetic dataset (setup (1)).}
     \end{subfigure}
     \centering
     \begin{subfigure}[b]{0.495\textwidth}
         \centering
         \includegraphics[width=\textwidth]{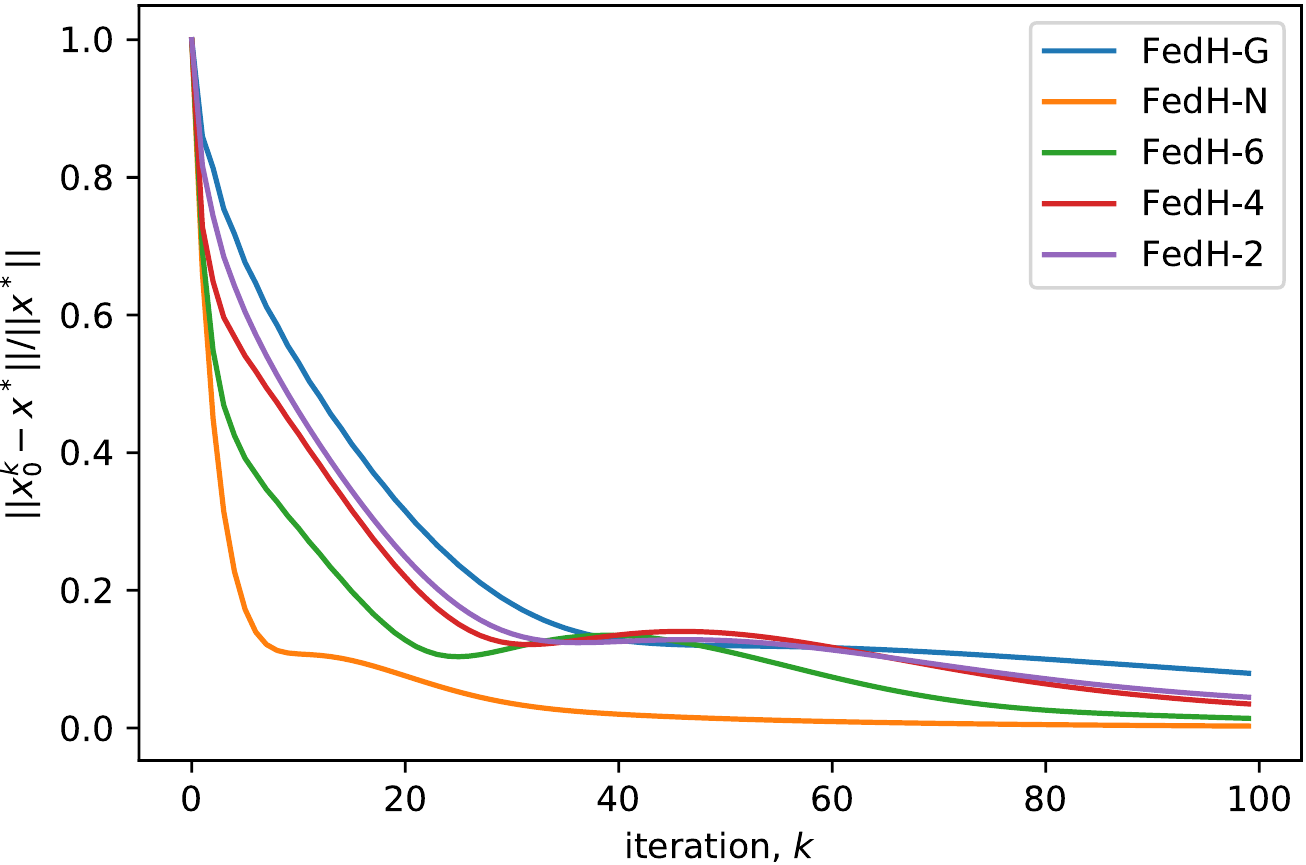}
         \caption{Least squares on Boston dataset (setup (2)).}
     \end{subfigure}
     \begin{subfigure}[b]{0.495\textwidth}
         \centering
         \includegraphics[width=\textwidth]{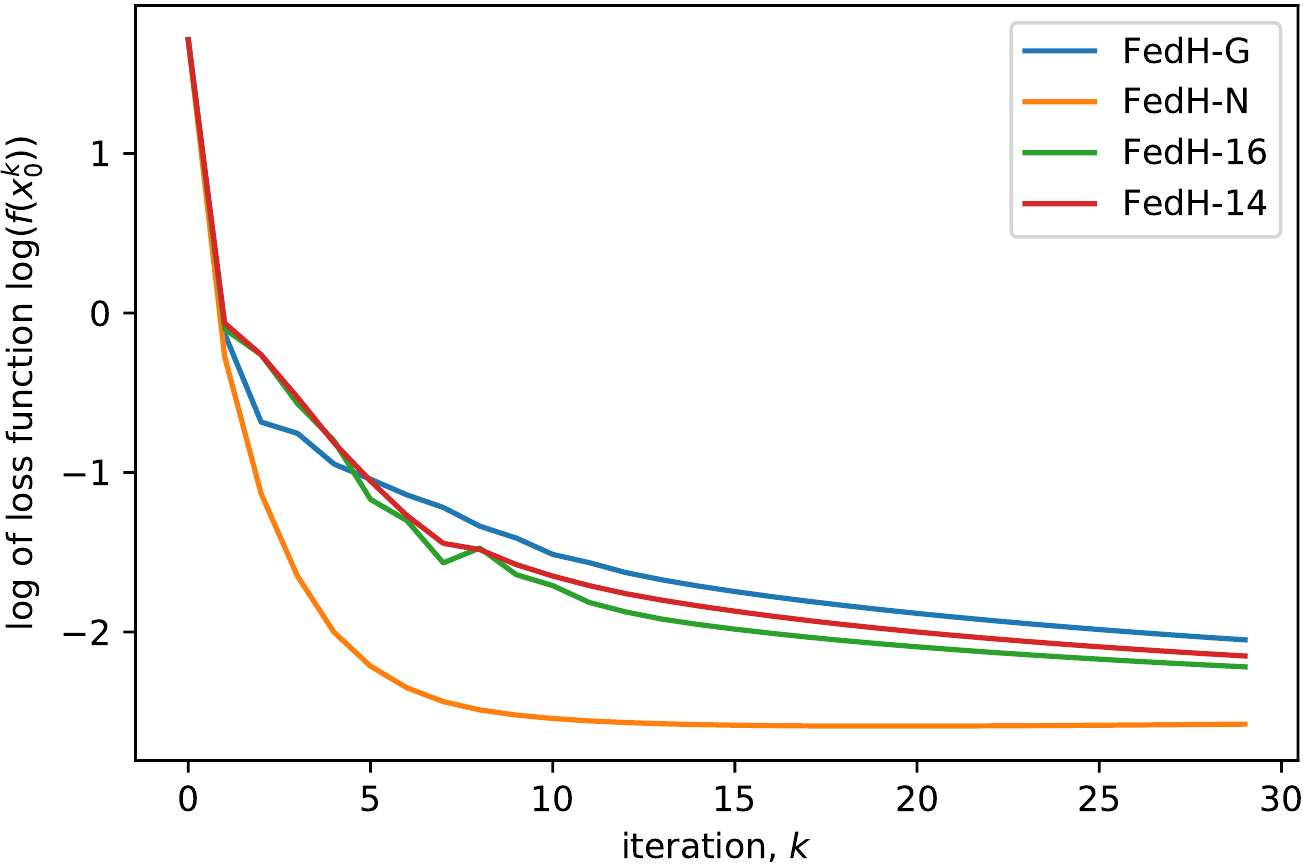}
         \caption{Logistic regression on Mushroom dataset (setup (4)).}
     \end{subfigure}
        \caption{Comparison of convergence rate with different number of Newton-type clients.}
        \label{fig:compare_2_app}
\end{figure}
\end{document}